\documentclass[11pt,a4paper]{amsart} 

\usepackage[T1]{fontenc}
\usepackage{lmodern}
\usepackage[margin=2.7cm]{geometry}

\usepackage{amssymb, mathrsfs, mathtools}

\usepackage[normalem]{ulem}
\usepackage{graphicx, enumitem, microtype, indentfirst}
\usepackage[dvipsnames]{xcolor}

\usepackage[hypertexnames=false, colorlinks=true, linkcolor=red!70!black,
            urlcolor=MidnightBlue, citecolor=MidnightBlue]{hyperref}
\usepackage[capitalize,nameinlink,noabbrev,compress]{cleveref}

\numberwithin{equation}{section}
\setlist[enumerate]{label=\textnormal{(\arabic*)}, itemsep=5pt, topsep=4pt, leftmargin=*}

\theoremstyle{plain}
\newtheorem{theo}{Theorem}[section]
\newtheorem{lem}[theo]{Lemma}
\newtheorem{prop}[theo]{Proposition}
\newtheorem{cor}[theo]{Corollary}
\theoremstyle{definition}

\theoremstyle{remark}
\newtheorem{rem}[theo]{Remark}

\crefname{theo}{Theorem}{Theorems}
\crefname{lem}{Lemma}{Lemmas}
\crefname{prop}{Proposition}{Propositions}
\crefname{cor}{Corollary}{Corollaries}
\crefname{defin}{Definition}{Definitions}
\crefname{rem}{Remark}{Remarks}
\crefname{equation}{}{}
\crefname{enumi}{}{}
\renewcommand{\eqref}{\cref}

\newcommand{\R}{\mathbb R}

\newcommand{\T}{\mathbb T}

\newcommand{\CC}{\mathcal C}

\newcommand{\HH}{\mathcal H}

\newcommand{\RR}{\mathcal R}

\newcommand{\id}{\mathrm{Id}}
\renewcommand{\d}{\mathrm{d}}
\newcommand{\eps}{\varepsilon}
\newcommand{\la}{\langle}
\newcommand{\ra}{\rangle}
\newcommand{\Nt}{|\hskip-0.04cm|\hskip-0.04cm|}

\renewcommand{\ge}{\geqslant}
\renewcommand{\leq}{\leqslant}
\renewcommand{\geq}{\geqslant}

\title[Fractional hypocoercivity in bounded domains]{Fractional hypocoercivity in bounded domains in the anomalous diffusion limit}

\author[M. Herda]{Maxime Herda}
\address[M. Herda]{Univ. Lille, CNRS, Inria, UMR 8524 - Laboratoire Paul Painlevé, F-59000 Lille, France}
\email{maxime.herda@inria.fr}
\author[M. Pegon]{Marc Pegon}
\address[M. Pegon]{Univ. Lille, CNRS, Inria, UMR 8524 - Laboratoire Paul Painlevé, F-59000 Lille, France}
\email{marc.pegon@univ-lille.fr}
\author[I. Tristani]{Isabelle Tristani}
\address[I. Tristani]{Université Côte d’Azur, CNRS, LJAD, Parc Valrose, F-06108 Nice, France} 
\email{isabelle.tristani@univ-cotedazur.fr}

\date{\today}

\begin{document}

\begin{abstract}In this paper, we provide quantitative stability results for several dissipative linear kinetic equations with heavy-tailed equilibria in different geometries. Moreover, the approach, inspired by the so-called $L^2$-hypocoercivity method, is robust enough to provide estimates that are uniform in the anomalous diffusion limit. It applies to bounded domains with either periodic or general Maxwell boundary conditions (ranging from the pure specular to the pure diffusive case) as well as to the whole-space case. In addition, our framework accommodates linear collisional operators whose coefficients depend on the position. 
\bigskip


\medskip


\end{abstract}

\maketitle

\tableofcontents

\section{Introduction}
\subsection{The problem}
We are interested in a class of linear kinetic equations describing the evolution of particles in the phase space $\Omega\times\R^d$ where the space domain $\Omega \subset \R^d$ can be bounded and connected with $\CC^2$ boundary or the whole space. Most of the paper is devoted to the case of a bounded domain, but our method is robust enough to also handle the whole-space case: the main ingredients are the same, although some additional work is required (see~\cref{sec:Rd}). Accordingly, the introduction mainly focuses on the bounded case, with boundary conditions. The periodic case then follows as an immediate corollary.

In the bounded case, we denote  by $n(x)$ the outward unit normal vector at $x \in \partial\Omega$. The boundary conditions, which will be defined below, take into account how the particles are reflected by the wall and take the form of a balance between the values of the trace on the outgoing and incoming velocity subsets of the boundary. Thus, we define the sets $\Sigma \coloneqq \partial\Omega \times \R^d$, as well as~$\Sigma_\pm^x \coloneqq \{ {v} \in \R^d; \pm \, {v} \cdot n(x) > 0 \}$ the sets of outgoing~($\Sigma_+^x$) and incoming~($\Sigma_-^x$) velocities at the point $x \in \partial\Omega$. We also introduce  
	\[
	\Sigma_\pm \coloneqq \Big\{ (x,{v}) \in \Sigma; \,\pm n(x) \cdot {v} > 0 \Big\} = \Big\{(x,{v}); \, x \in \partial\Omega, \, {v} \in \Sigma^x_\pm \Big \}\,. 
	\]
We denote by $\gamma f$ the trace of $f$ on $\Sigma$, and by $\gamma_{\pm} f = \mathbf 1_{\Sigma_{\pm}} \gamma f$ the traces on $\Sigma_{\pm}$ ($\gamma_+$ corresponding to outgoing velocities and $\gamma_-$ to incoming velocities). To lighten the notation, we shall use the abbreviation $f_\pm \coloneqq \gamma_\pm f$.

Let $\eps >0$. 
For a density function~$f^\eps=f^\eps(t,x,v)$, $t \ge 0$, $x \in \Omega$, $v \in \R^d$, we consider the following equation 
	\begin{equation} \label{eq:dtf=Lf}
	\left\{
	\begin{array}{lllll}
		\partial_t f^\eps &=& \Lambda^\eps f^\eps \coloneqq - \eps^{1-2s} v \cdot \nabla_{\!x} f^\eps +\eps^{-2s} L f^\eps &\text{in} \quad &\R^+ \times \Omega \times \R^d \\[0,1cm] 
		\gamma_{\!-} f^\eps &=& \RR \gamma_{\!+} f^\eps &\text{on} \quad &\R^+ \times \partial \Omega \times \R^d \,, 
	\end{array}
	\right. 
	\end{equation}
where $L$ and $\RR$ stand for two linear collisional operators, all of which will be described in detail below. 
Let us mention that the trace functions $\gamma_\pm f$ are well defined in our framework (see \cref{subsec:notation}).

Our goal is to investigate the long-time behavior of solutions to linear equations with heavy-tailed equilibria. More precisely, in~\eqref{eq:dtf=Lf}, the kernel of the collision operator $L$ is spanned by a distribution $M$ that only depends on $v$ and has infinite energy. Our study includes a linear Boltzmann operator as well as a L\'evy--Fokker--Planck operator that both have dissipativity properties in velocity. Before going into the description of the main results and strategy of our proof, let us recall that results on long-time behavior of solutions to the space-homogeneous version of~\eqref{eq:dtf=Lf} which simply reads~$\partial_tf(t,v) = L f(t,v)$, have been obtained in~\cite{Gentil-Imbert} in spaces of type~$L^2\big(M^{-1} \d v\big)$ (among others) and later in~\cite{Tristani-CMS} in larger Lebesgue spaces (in the case where $L$ is a L\'evy--Fokker--Planck operator, see~\eqref{def:LFP}). Notice that the presence of the transport operator in our equation~\eqref{eq:dtf=Lf} makes the analysis more intricate and to handle this, we will rely on hypocoercivity techniques. Hypocoercivity methods allow one to recover dissipativity in the whole phase space and rely on the understanding of the interplay between the transport operator~$-v \cdot \nabla_{\!x}$ which is conservative and the linear collision operator $L$ which is dissipative in the velocity variable (see~\eqref{eq:coercivity}). The goal is to construct a Lyapunov functional for the whole problem. 
Note that the transport and the collision operators do not commute, which gives nontrivial mixing properties, and that the transport operator does not preserve collision invariants. The macroscopic quantities are thus mixed inside the spatial domain by the transport operator. Thanks to this type of method, the long-time behavior of solutions to~\eqref{eq:dtf=Lf}  has so far been studied only on the torus and in the whole space, see~\cite{BDL,AHHT1}. 
While we derive uniform-in-$\eps$ estimates, the goal of the present paper is not to study the anomalous diffusion limit~$\eps\to0$ in bounded domains. For details about this limit, we refer to earlier works, non-exhaustively  \cite{mellet2011fractional, benabdallah2011fractional, cesbron2018anomalous, cesbron2022fractional,dechicha2024fractional} and references therein. 

To the best of our knowledge, this is the first study of the long-time behavior of kinetic equations with heavy-tailed equilibria in a bounded domain with general Maxwell boundary conditions. Our estimates are uniform in the anomalous diffusion limit, the rate is explicit in the size of the domain, and the same functional covers the torus and the whole space.

To end this part, we emphasize that our approach is inspired by previous works on~$L^2$-hypocoercivity but the method has to be adapted to heavy-tailed equilibria in the diffusion limit. For heavy-tailed equilibria, our main source of inspiration is~\cite{BDL}, where fractional hypocoercivity is developed through a mode-by-mode analysis in Fourier variables. Since the Fourier transform in~$x$ is not available in a general bounded domain, we work in physical variables. In bounded domains and on the torus, our proofs do not use the Fourier transform, and in the whole space it only enters in the lower bound on the gain term (see \cref{lem:compare_gain_Fourier,lem:PhiG}). Working in physical variables also allows us to deal with linear collisional operators (see~\eqref{def:L1} and~\eqref{def:LFP}) whose coefficients depend on the position~$x$. We refer to \cref{subsec:main} for more comments on the result and on the method of proof. We now focus on the model under study.

\medskip
\noindent 
{\it The collision operator.} 
Throughout the paper, we will consider two types of collisional operators. The first one will be a so-called {\em linear Boltzmann} operator. We introduce the distribution $M_1$ defined by 
	\begin{equation} \label{def:M1}
	M_1(v) \coloneqq \frac{c_{d,s}}{\langle v \rangle^{d+2s}}\,,\quad s \in \left(0,1\right)
	\end{equation}
where $\langle v\rangle^2 := 1+|v|^2$, the positive constant $c_{d,s}$ is such that $\int_{\R^d} M_1(v) \, \d v = 1$
and we consider the following linear operator 
	\begin{equation} \label{def:L1}
	L_1g(x,v)  \coloneqq \int_{\R^d} \zeta(x,v,v') \left[g(x,v')M_1(v) - g(x,v) M_1(v')\right] \, \d v'
	\end{equation}
where $\zeta :\Omega \times \R^d \times \R^d \to \R^+$ is measurable and such that $\zeta(x,v,v') = \zeta(x,v',v)$ for any~$(x,v,v')$. We furthermore assume that there exist $\zeta_0>0$ and $\zeta_1>0$ such that 
	\begin{equation} \label{eq:bound-sigma}
	\zeta_0 \leq \zeta(x,v,v') \leq \zeta_1 \,, \quad \forall \, (x,v,v') \in \Omega \times \R^d \times \R^d\,. 
	\end{equation}
We have $\operatorname{Ker} L_1 = \operatorname{Span} M_1$ and it is worth noticing that when $\zeta \equiv 1$, we recover the simplest form of the linear Boltzmann operator: 
	\begin{equation} \label{def:BGK}
	L_1 g(x,v) = \rho[g](x) M_1(v) - g(x,v) \qquad \text{with} \qquad \rho[g](x) \coloneqq \int_{\R^d} g(x,v) \, \d v\,.
	\end{equation}
The second operator that we shall consider is a {\em L\'evy--Fokker--Planck} operator defined through 
	\begin{equation} \label{def:LFP}
	L_2 g(x,v) \coloneqq \nu(x) \left[- (-\Delta_v)^s g(x,v) + \nabla_{\!v} \cdot (vg(x,v))\right]\,, \quad  s \in \left(0,1\right)\,,
	\end{equation}
where $\nu : \Omega \to \R^+$ is measurable. We furthermore assume that there exist $\nu_0,\nu_1>0$ such that 
    \begin{equation} \label{eq:bound-nu}
    \nu_0 \leq \nu(x) \leq \nu_1\,, \quad \forall\,x \in \Omega \,. 
    \end{equation}
In this case, it is well-known that $\operatorname{Ker} L_2 = \operatorname{Span} M_2$, where $M_2$ is radially symmetric and such that $\int_{\R^d} M_2(v) \, \d v =1$, with the following bounds:
	\begin{equation} \label{eq:bound-M2}
	\langle v \rangle^{-d-2s} \lesssim M_2(v) \lesssim \langle v \rangle^{-d-2s}\,, \quad \forall \, v \in \R^d\,. 
	\end{equation}
For this, we refer to~\cite{AHHT1,AHHT2} and references therein. 

Since we are going to treat both cases in a unified framework, we shall drop the indices in the notation: we denote by $L$ the operator $L_1$ or $L_2$ indifferently and by $M$ the distribution~$M_1$ or~$M_2$ of mass $1$ that spans the kernel of $L_1$ or $L_2$. 

\medskip
\noindent 
{\it The reflection operator.} Let us introduce $\alpha : \partial\Omega \to [0,1]$ a Lipschitz function called the accommodation coefficient and assume from now on that  
\begin{equation}\label{hyp:accomodation}
\text{either }\Bigl(\frac12<s<1\Bigr)\ \text{ or }\  \Bigl(0<s<1\text{ and }\alpha \equiv 0\Bigr).
\end{equation}
We assume that the reflection operator acts locally in time and position, namely
	\begin{equation} \label{def:reflection}
	(\RR \gamma_{\!+}  g)(t,x,v) =  \RR_x (\gamma_{\!+}  g (t,x,\cdot))(v)
	\end{equation}
and more specifically it is a possibly position-dependent Maxwell boundary condition operator 
	\begin{equation}\label{eq:boundary} 
	\RR_x (g (x,\cdot))(v)  \coloneqq (1-\alpha(x))   g (x , R_x v) + \alpha(x) D g  (x,v)\,,
	\end{equation}
for any $(x,v) \in \Sigma_-$ and for any function $g : \Sigma_+ \to \R$. Here $R_x$ is the specular reflection operator
	\[
	R_x v \coloneqq v - 2 n(x) (n(x) \cdot v)\,,
	\]
and $D$ is the diffusive operator
	\begin{align} \label{eq:def_D}
	D g(x,v) \coloneqq c_M M(v) \bar g (x) 
	\qquad \text{with} \qquad
	\bar g (x) \coloneqq \int_{\Sigma^x_+} g(x,w) \, n(x) \cdot w \, \d w \,,
	\end{align}
where the constant  $c_M$ is such that  $c_M \bar{M} = 1$ and we recall that $M$ is the distribution of mass $1$ that spans the kernel of the collision operator $L$. 
\begin{rem}
Notice that the restriction $s>1/2$ in the models that we consider is only necessary for $\bar M$ to be well defined. This restriction can be dropped if $\alpha \equiv 0$ which explains hypothesis \eqref{hyp:accomodation}. This assumption is thus not needed in the cases of the torus and the whole space for which we can consider the whole range $s \in (0,1)$.
\end{rem}
The boundary condition \eqref{eq:boundary} corresponds to the \emph{pure specular reflection} boundary condition when  $\alpha \equiv 0$ 
and it corresponds to the \emph{pure diffusive} one when~$\alpha \equiv 1$. 

One can easily prove that when $g$ satisfies $\gamma_- g = \RR \gamma_+g$ with $\RR$ given in~\eqref{eq:boundary}, we have 
	\[
	\int_{\R^d} \gamma g \, n(x) \cdot v \, \d v = 0 \,, \quad \forall \,x \in \partial \Omega\,.
	\]
This in particular implies that if $f^\eps$ is a solution to~\eqref{eq:dtf=Lf}, then 
	\[
	\frac{\d}{\d t} \int_{\Omega \times \R^d} f^\eps(t,x,v) \, \d v \, \d x = 0\,. 
	\]
By linearity, we may thus assume that $\int_{\Omega \times \R^d} f^\eps(t,x,v) \, \d v \, \d x = 0$ for any $t
\geq 0$, that is,~$\rho[f^\eps]$ is zero-averaged.

To conclude this discussion on the boundary condition, we note that hypocoercivity tools do not appear well suited to handle diffusive boundary conditions when the velocity distribution at the boundary differs from the interior equilibrium. Indeed, these methods rely on knowing the global equilibrium distribution and achieving some separation of variables, as the velocity distribution is used to define suitable weighted functional spaces. This type of configuration is studied in~\cite{Bernou} using probabilistic tools (Doeblin and Harris theorems). The author establishes the existence of an equilibrium for the equation but obtains no information on its form (in particular, its behavior in the boundary layer). Nevertheless, a stability result for this equilibrium is obtained (though the analysis is restricted to the case $\eps=1$).
\subsection{Functional framework and notation} \label{subsec:notation}
First, we introduce the weighted Hilbert space~$\HH \coloneqq L^2(M^{-1} \, \d v \, \d x)$ which is a natural space to study the linear operator $L$ (see~\cref{subsec:coercivity}). As mentioned above, we have $\operatorname{Ker} L = \operatorname{Span} M$. Moreover, the projection~$\pi$ onto the kernel of $L$ is given by $\pi g = \rho[g] M$ where $\rho[g]$ is the local mass of $g$ introduced in~\eqref{def:BGK}. We also define the microscopic part of $g$ by $g^\perp \coloneqq g - \pi g = g - \rho[g] M$. The term microscopic is justified by the link with the diffusion limit problem.

We also introduce the following notation: for a suitable distribution $g$, we define the operator $D^\perp \coloneqq \mathrm{Id} - D$, where $D$ is given by~\eqref{eq:def_D} and  the following space 
	\[
	\partial \HH_+ \coloneqq L^2\big(\Sigma_+ ; M^{-1}(v) n(x) \cdot v \, \d v \, \d\mathscr{H}^{d-1}(x)\big)\,,
	\] 
where we recall that $\mathscr{H}^{d-1}$ is the $(d-1)$-dimensional Hausdorff measure restricted to $\partial \Omega$. Concerning the definition of $\Lambda^\eps$ with the boundary condition, it is first defined on
\[
 \mathcal D\coloneqq\Big\{f\in\HH:\ v\cdot\nabla_{\!x}f\in\HH,\ Lf\in\HH,\ \gamma_+f\in\partial\HH_+,\ \gamma_-f=\RR\gamma_+f\Big\}\,,
 \]
where we point out that the trace functions $\gamma_\pm f$ are well defined adapting results from~\cite{Bardos,Cessenat,Mischler-CPDE}. Since $\mathcal D$ is dense in $\HH$ and $\Lambda^\eps$ is dissipative from \cref{lem:coercivity}, we can use~\cite[Theorem~4.5-(c)]{Pazy} to deduce that $\Lambda^\eps$ is closable and we still denote by $\Lambda^\eps$ its closure and by ${\rm Dom}(\Lambda^\eps) \supset \mathcal D$ its domain. All the computations made throughout the paper are licit for elements of $\mathcal D$ and the main results (that are stated in ${\rm Dom}(\Lambda^\eps)$) follow by density.

It is worth pointing out that the notation $D^\perp$ is justified by the fact that for any $f$ and $g$ in~$\mathcal D$, we have~$Df(x,\cdot) \perp D^\perp g(x,\cdot)$ in $L^2(\Sigma^x_+; M^{-1}(v)n(x)\cdot v\,\d v)$ for $\mathscr{H}^{d-1}$-a.e. $x\in\partial\Omega$. Indeed, a simple computation shows that 
	\begin{align*}
	&\int_{\Sigma^x_+} Df(x,v) D^\perp g(x,v) M^{-1}(v) \, n(x) \cdot v \, \d v \\
	&\qquad 
	= \int_{\Sigma^x_+} c_M \bar f(x) \, g(x,v) \, n(x) \cdot v \, \d v
	- \int_{\Sigma^x_+} c_M^2 M(v) \, \bar f(x) \, \bar g(x) \, n(x) \cdot v \, \d v= 0
	\end{align*}
where we used the fact that $c_M \bar M =1$ to get the last equality.

In what follows, we shall keep track of how the multiplicative constants involved in our estimates depend on the diameter of~$\Omega$ which we denote by
\begin{equation}\label{eq:defdiam}
\d_\Omega \coloneqq \mathrm{diam}(\Omega).
\end{equation}
We treat the whole-space case where the diameter is infinite in \cref{sec:Rd}. We also introduce the normalized bounded open set~$\widehat \Omega \coloneqq \d_\Omega^{-1}\Omega $.
Finally, we also introduce $\mu_1(\Omega)$ the first non-zero eigenvalue of the Neumann Laplacian on $\Omega$. Notice that by a scaling argument, it is straightforward to prove that $\mu_1(\Omega) = \mu_1(\widehat \Omega) \, \d_\Omega^{-2}$.

To end this part, we mention that throughout the paper, unless stated otherwise, $C$ denotes a positive constant that may change from line to line and depends only on fixed numbers (but not on $\d_\Omega$ or $\eps$). We may also write $A \lesssim B $ when $A\leq CB$ for such a constant~$C$.

\subsection{Main results} \label{subsec:main}
The main result that we shall prove in this paper is the following:
\begin{theo}\label{theo:main}
Let $\eps_0>0$. For any $\eps \in (0,\eps_0\d_\Omega)$, there exists a scalar product $\la\!\la \cdot , \cdot \ra \! \ra_{\eps}$
on the space~$\HH$ so that the associated norm $\Nt \cdot \Nt_\eps$ is equivalent to the usual norm
$\| \cdot \|_{\HH}$ uniformly in~$\eps$, and for which the linear operator
$\Lambda^\eps$ satisfies the following coercivity estimate:
$$
\la \! \la \Lambda^\eps f , f \ra\!\ra_{\eps} 
\lesssim - \frac{1}{\d_\Omega^{2s}} \Nt f \Nt^2_\eps - \frac{1}{\eps^{2s}} \|f^\perp\|_{\HH}^2- \frac{1}{\eps^{2s-1}}\Big\|\sqrt{\alpha(2-\alpha)} D^\perp f_+\Big\|^2_{\partial \HH_+} \,, 
$$
for any $f \in \mathrm{Dom}(\Lambda^\eps)$ satisfying the boundary condition in~\eqref{eq:dtf=Lf}
and $\int_{\Omega \times \R^d} f(x,v) \, \d v \, \d x =0$.
The multiplicative constant in $\lesssim$ depends only on $\eps_0$, $\widehat\Omega$, $s$, $d$, $\zeta$, $\nu$ (defined in \eqref{def:L1} and~\eqref{def:LFP}) and the coercivity constant~$\lambda_0$ (defined in~\cref{eq:coercivity}).
\end{theo}

From this result, we immediately deduce the following stability result:
\begin{cor} \label{cor:main}
Let $\eps_0>0$.
For any $\eps \in (0,\eps_0\d_\Omega)$, we consider $f_{\rm in}^\eps \in \HH$ with vanishing mass, i.e., satisfying~$\int_{\Omega \times \R^d} f_{\rm in}^\eps(x,v) \, \d
v \, \d x =0$. There exist constants $c,C>0$ depending only on $\eps_0$, $\widehat\Omega$, $s$, $d$, $\zeta$, $\nu$ (defined in \eqref{def:L1} and~\eqref{def:LFP}) and the coercivity constant~$\lambda_0$ (defined in \cref{eq:coercivity}) such that, for any solution $f^\eps$ to~\eqref{eq:dtf=Lf} associated with~$f_{\rm in}^\eps$ and for any~$t \geq 0$, there holds
\[
\| f^\eps(t)\|_\HH \leq C \exp\Big(-c \, \d_\Omega^{-2s} t\Big) \| f_{\rm in}^\eps \|_\HH\,,
\]
\[
\frac{1}{\eps^{2s}} \int_0^\infty \|(f^\eps)^\perp(\tau)\|^2_\HH \, \exp\Big(2 c \, \d_\Omega^{-2s}  \tau\Big) \, \d\tau \leq C \|f_{\rm in}^\eps\|^2_\HH\,,
\] 
and
\[
\frac{1}{\eps^{2s-1}} \int_0^\infty \|\sqrt{\alpha(2-\alpha)} D^\perp f_+^\eps\Big\|^2_{\partial \HH_+} \, \exp\Big(2 c \, \d_\Omega^{-2s}  \tau\Big) \, \d\tau \leq C \|f_{\rm in}^\eps\|^2_\HH\,.
\]
\end{cor}

\begin{rem}
The exact same results can be obtained if $\Omega = \T^d_\ell$ the $d$-dimensional torus of size $\ell>0$. The adaptation is straightforward (see \cref{sec:torus}). 
\end{rem}

We also present a result in the whole space. When the collision operator does not depend on~$x$, the corresponding decay estimate was obtained for $\eps = 1$ in~\cite{BDL}, for a broader class of collision operators. We prove it with the same functional as in the bounded case, so that it also holds when $\zeta$ and $\nu$ depend on~$x$, a case that is not covered by the Fourier approach.

\begin{theo} \label{theo:wholespace}
Let $\eps_0>0$. For any $\eps\in(0,\eps_0)$, there exists a scalar product $\la\!\la \cdot , \cdot \ra \! \ra_{\eps}$
on the space~$\HH$ so that the associated norm $\Nt \cdot \Nt_\eps$ is equivalent to the usual norm
$\| \cdot \|_{\HH}$ uniformly in~$\eps$, and for which the linear operator
$\Lambda^\eps$ satisfies the following  estimate:
	\[
\la \! \la \Lambda^\eps f , f \ra\!\ra_{\eps} 
\lesssim - \Nt f \Nt_\eps^{2}\left(\frac{\Nt f \Nt_\eps^{\frac{4s}{d}}}{\|\rho[f]\|_{L^1(\R^d)}^{\frac{4s}{d}}+\Nt f
\Nt_\eps^{\frac{4s}{d}}}\right)
- \frac{1}{\eps^{2s}} \|f^\perp\|_{\HH}^2\,,
\]
for any $f \in \mathrm{Dom}(\Lambda^\eps) \cap L^1(\R^d \times \R^d)$.
The multiplicative constant involved in $\lesssim$ depends only on~$\eps_0$, $s$, $d$, $\zeta$, $\nu$ (defined in \eqref{def:L1} and~\eqref{def:LFP}) and the coercivity constant~$\lambda_0$ (defined in~\cref{eq:coercivity}).
\end{theo}

A consequence of this hypocoercivity estimate is the following:
\begin{cor} \label{cor:wholespace}
Let $\eps_0>0$. For $\eps \in (0,\eps_0)$ we consider $f_{\rm in}^\eps \in L^1(\R^d \times \R^d)
\cap \HH$.
Then there exists a constant $C>0$ depending only on $\eps_0$, $s$, $d$, $\zeta$, $\nu$ (defined in \eqref{def:L1} and~\eqref{def:LFP}) and the coercivity constant~$\lambda_0$ (defined in \cref{eq:coercivity})
such that any solution $f^\eps$ to
    \[
    \partial_t f^\eps = \Lambda^\eps f^\eps \quad \text{in} \quad \R^+ \times \R^d \times \R^d
    \]
associated with~$f_{\rm in}^\eps$ satisfies the following decay estimates: for any $t \geq 0$,
    \[
    \|f^\eps(t)\|_\HH \leq 
	\frac{ C\|f_{\rm in}^\eps\|_\HH }{\left(1+\mu_{\rm in}^\eps t\right)^{\frac{d}{4s}}}\,,\quad \text{where}\quad\ \mu_{\rm in}^\eps \coloneqq \frac{\|f_{\mathrm{in}}^\eps\|_\HH^{\frac{4s}{d}}}{\|f_{\rm in}^\eps\|_{L^1}^{\frac{4s}{d}}+ \|f_{\mathrm{in}}^\eps\|_\HH^{\frac{4s}{d}}}\,
    \]
and, for every $\beta\in(0,\tfrac{d}{2s})$,
    \[
    \frac{1}{\eps^{2s}} \int_0^\infty \|(f^\eps)^\perp(\tau)\|^2_\HH \,  (1+\mu_{\rm in}^\eps \tau)^{\beta} \,
	\d\tau \leq \frac{C}{\frac{d}{2s}-\beta}
\|f_{\rm
	in}^\eps\|^2_\HH\,.
    \]
\end{cor}

Let us give more comments on fractional hypocoercivity. One of its main challenges is that the equilibrium of our equation~$M$ has infinite second order moments (see~\eqref{def:M1} and~\eqref{eq:bound-M2}). This prevents us from using standard $L^2$-hypocoercivity approaches developed in~\cite{Herau-AA,DMS1,DMS2,BCMT} for example. The difficulty was overcome in~\cite{BDL}, where fractional $L^2$-hypocoercivity was established with a decay rate that matches that of the fractional heat equation. A key idea of~\cite{BDL} which we also use, is to introduce velocity weights in the hypocoercivity functional (here in the norm $\Nt \cdot \Nt_\eps$ defined in~\eqref{def:Nt}), to counterbalance the lack of integrability of the equilibrium. 
Concerning the treatment of the boundaries, our approach is based on tools (that have to be refined in our case) introduced in~\cite{BCMT,Carrapatoso-Mischler}. Indeed, our new norm involves the inverse of an operator that has to be defined conveniently to deal with the boundary terms and the lack of decay of our data.
In particular, we need sharper trace estimates than the mere continuity of the trace operator. We recall them in~\cref{app:trace}. 
The two works closest to ours are thus~\cite{BCMT}, where uniform-in-$\eps$ hypocoercivity estimates are obtained in bounded domains with general Maxwell boundary conditions, in the classical diffusion scaling with Gaussian equilibria, and~\cite{BDL}, where heavy-tailed equilibria are treated on~$\T^d$ and~$\R^d$. 
 The novelties of the present paper are the following:
\begin{enumerate}
\item heavy-tailed equilibria and the anomalous diffusion scaling are handled in bounded domains with general Maxwell boundary conditions;
\item the decay rate is explicit in the size of the domain: it is of order $\d_\Omega^{-2s}$, uniformly in~$\eps\in(0,\eps_0\d_\Omega)$;
\item the collision operators may depend on the position, through $\zeta$ and $\nu$, also in the whole space;
\item the same functional is used in bounded domains, on the torus and in the whole space.
\end{enumerate}

Our approach is
reminiscent of the one used for the hydrodynamic limit of the non-cutoff Boltzmann equation in the torus and in the whole space in
\cite{CaoCarrapatoso2026}, in its corrected form \cite{corrigendum} and of the one used in~\cite{BCMT} to get hypocoercivity in bounded domains for linearized Boltzmann and Landau operators. On the one hand, in~\cite{corrigendum}, the norm is defined mode by
mode in Fourier and involves the multiplier $\eps\xi/\la\eps\xi\ra^{2}$, i.e., the symbol of
$\eps\nabla_{\!x}\,(\id-\eps^2\Delta_x)^{-1}$. On the other hand, in~\cite{BCMT}, the new norm involves the operator $\eps \nabla_{\!x} \Delta_x^{-1}$.
Our definition of the new norm thus differs from that in~\cite{corrigendum,BCMT}: we use the inverse of $\id - \sigma^2 \Delta_x$ for $\sigma>0$ (see~\cref{subsec:elliptic}), with a scale that depends on the velocity, namely $\sigma = \eps \langle v \rangle$. This operator with velocity-dependent scale, which is reminiscent of \cite{BDL}, has to be carefully defined in bounded domains through an elliptic problem which is discussed in \cref{subsec:elliptic}. The quantity $\sigma = \eps \langle v \rangle$ also appears naturally when computing the (anomalous) diffusion limit. In the case of a Gaussian equilibrium, one can close the estimates by choosing $\sigma = \eps \langle v \rangle^p$ for any~$p \geq 0$. However, in our case of heavy-tailed equilibria, the choice $p=1$ is optimal. Indeed, the lack of integrability of our distributions forbids $p<1$, and any $p>1$ yields a nonuniform-in-$\eps$ rate of decay.  

Let us mention that in \cref{app:alternative}, we present an alternative way to define a hypocoercive norm in bounded domains. However, this approach, though a bit simpler, is not robust enough to deal with the whole-space case. Notice that this alternative approach is still more involved than the one presented in~\cite{BCMT} because we need to introduce an $\eps$-dependency in the definition of the new norm but also in the definition of our ``twisted'' macroscopic quantities, see~\eqref{def:twisted}. The introduction of these new macroscopic quantities is motivated by a heuristic stemming from the study of the fractional diffusion limit. To end this part, we sketch why the method presented in \cref{app:alternative} is not adapted to the whole-space setting. It fails to capture the fractional diffusive behavior of the dissipation term that arises at low frequencies (see~\cite{BDL}). As a result, it does not recover the optimal decay rate obtained in~\cite{BDL}. This is because, in this alternative functional, apart from  the~$\eps$-dependencies, the inverse Laplacian and the velocity average operators are applied in a different order than in~\cite{BDL}. In the method presented in \cref{sec:proof}, we overcome this problem by introducing the operator~$(\id- \eps^2 \langle v \rangle^2 \Delta_x)^{-1}$, which depends on the velocity variable and is applied to an averaged quantity in velocity.

\subsection{Plan of the paper}
\Cref{sec:prelim} provides the main tools needed to handle the case where $\Omega$ is a bounded domain with general Maxwell boundary condition. In~\Cref{subsec:newnorm}, we introduce our new norm that shall be hypocoercive. \Cref{subsec:proof} is devoted to the proof of our main result in the bounded case. In \cref{sec:torus} and~\cref{sec:Rd}, we explain the slight changes that need to be made to adapt the proof to the cases of the torus and the whole space. In~\cref{app:trace}, we present trace inequalities with the corresponding proofs. In~\cref{app:alternative}, we provide an alternative hypocoercive functional for bounded domains.

\section{Preliminaries} \label{sec:prelim}
\subsection{Elliptic problems} \label{subsec:elliptic}

We consider $\eta_1\in L^2(\Omega)$, $\eta_2\in H(\mathrm{div}, \Omega)$ and a positive constant~\mbox{$\sigma>0$}.
Recall that any $\eta_2\in H(\mathrm{div},\Omega)$ admits a normal trace $\eta_2\cdot n \in H^{-1/2}(\partial\Omega)$ (see~\cite{Tem1979}). We then define
$u_\sigma \in H^1(\Omega)$ as the solution to the variational problem 
	\begin{multline} \label{eq:var}
	\int_\Omega u_\sigma w\,\d x+\sigma^2\int_\Omega \nabla_{\!x}\,u_\sigma \cdot \nabla_{\!x} w \, \d x\\
	= \int_{\Omega} \left(w \, \eta_1 - \nabla_{\!x} w \cdot \eta_2 \right)\, \d x\,
	+\langle \eta_2\cdot n, w\rangle_{\partial\Omega}\,, \quad \forall w \in H^1(\Omega)\,,
	\end{multline}
where $\langle \eta_2\cdot n, w\rangle_{\partial\Omega}$ denotes the duality pairing between $H^{-1/2}(\partial\Omega)$ and $H^{1/2}(\partial\Omega)$.
This is exactly the variational solution to the following equation with Neumann boundary condition:
	\[
	(\id-\sigma^2\Delta_x) u_\sigma = \eta_1 + \operatorname{div}_x \eta_2 \quad \text{in} \quad \Omega\,, \quad 
	n(x) \cdot \nabla_{\!x}\,u_\sigma = 0 \quad \text{on} \quad \partial\Omega\,.
	\]
This variational problem can be solved thanks to the Lax--Milgram theorem. 
In what follows, we shall use the notation $u_\sigma = (\id-\sigma^2\Delta_x)^{-1} (\eta_1 +
\operatorname{div}_x \eta_2)$.
From \cref{eq:var}, it is clear that $(\id-\sigma^2\Delta_x)^{-1}$ is self-adjoint in $L^2(\Omega)$.

\begin{lem}[Case $\eta_1=0$]
For $\eta_1=0$ and $\eta_2\in  H(\mathrm{div}, \Omega)$ such that $\eta_2\cdot n\in L^2(\partial\Omega)$, the solution $u_\sigma$ of~\cref{eq:var} satisfies
\begin{equation}\label{eq:ellipticeta2}
\|\nabla_{\!x}\, u_\sigma\|_{L^2(\Omega)} \leq C(\widehat\Omega)\left(
\sigma^{-2}\|\eta_{2}\|_{L^2(\Omega)}+
\sigma^{-\frac{3}{2}}\|\eta_{2}\cdot n\|_{L^2(\partial\Omega)}\right)\,.
\end{equation}
\end{lem}

\begin{proof}
In the following, $C(\widehat\Omega)$ may change from line to line, but depends only on $\widehat\Omega$. Taking $w=u_\sigma-\langle u_\sigma\rangle$ in \cref{eq:var}, where $\langle u_\sigma\rangle$ denotes the average of $u_\sigma$ over $\Omega$, and using the fact that $\langle u_\sigma, u_\sigma-\langle u_\sigma\rangle\rangle_{L^2(\Omega)} = \|u_\sigma-\langle u_\sigma\rangle\|_{L^2(\Omega)}^2$, we estimate
\begin{multline}\label{lemeta1:eq1}
\|u_\sigma-\langle u_\sigma\rangle\|_{L^2(\Omega)}^2+\sigma^2\|\nabla_{\!x}\,u_\sigma\|_{L^2(\Omega)}^2
= -\langle \eta_2, \nabla_{\!x}\, u_\sigma\rangle_{L^2(\Omega)}
+ \langle u_\sigma-\langle u_\sigma\rangle, \eta_2\cdot n\rangle_{L^2(\partial\Omega)}\\
\leq \|\eta_2\|_{L^2(\Omega)}\|\nabla_{\!x}\,u_\sigma\|_{L^2(\Omega)}
+\|u_\sigma-\langle u_\sigma\rangle\|_{L^2(\partial\Omega)} \|\eta_2\cdot n\|_{L^2(\partial\Omega)} \,.
\end{multline}
By the multiplicative trace inequality of~\cref{prop:multtraceineq} and the Poincaré--Wirtinger inequality, we have
\[
\begin{aligned}
\|u_\sigma-\langle u_\sigma\rangle\|_{L^2(\partial\Omega)}
&\leq C(\widehat\Omega)\|u_\sigma-\langle u_\sigma\rangle\|_{L^2(\Omega)}^{\frac{1}{2}}
\left(\d_\Omega^{-1} \|u_\sigma-\langle u_\sigma\rangle\|_{L^2(\Omega)}+\|\nabla_{\!x}\, u_\sigma\|_{L^2(\Omega)}\right)^{\frac{1}{2}}\\
& \leq C(\widehat\Omega) \|u_\sigma-\langle u_\sigma\rangle\|_{L^2(\Omega)}^{\frac{1}{2}}
\|\nabla_{\!x}\, u_\sigma\|_{L^2(\Omega)}^{\frac{1}{2}}.
\end{aligned}
\]
Inserting this into \cref{lemeta1:eq1} and using Young's inequality, we deduce
\begin{multline*}
\|u_\sigma-\langle u_\sigma\rangle\|_{L^2(\Omega)}^2
+\sigma^2\|\nabla_{\!x}\,u_\sigma\|^2_{L^2(\Omega)}\\
\leq \frac{1}{2}\|u_\sigma-\langle u_\sigma\rangle\|_{L^2(\Omega)}^2+\frac{\sigma^2}{2}\|\nabla_{\!x}\, u_\sigma\|_{L^2(\Omega)}^2 + C(\widehat\Omega)\left(
\sigma^{-2}\|\eta_2\|_{L^2(\Omega)}^2
+\sigma^{-1} \|\eta_2\cdot n\|_{L^2(\partial\Omega)}^2\right).
\end{multline*}
This implies \cref{eq:ellipticeta2}, which concludes the proof.
\end{proof}

In the next proposition, recall that $\mu_1(\Omega)$ stands for the first non-zero eigenvalue of the Neumann Laplacian on $\Omega$.
\begin{lem}[Case $\eta_2=0$]\label{prp:ellipticeta1}
For $\eta_1\in L^2(\Omega)$ and $\eta_2=0$, let $u_\sigma$ be the solution of~\cref{eq:var}
and set $E:=\sqrt{\langle \eta_1-u_\sigma,\eta_1\rangle}$. Then $E$ is well defined, with
\[
E\leq \|\eta_1\|_{L^2(\Omega)}\,,
\]
and $u_\sigma$ satisfies
\begin{equation}\label{eq:ellipticlapu}
\|\Delta_x u_\sigma\|_{L^2(\Omega)}\leq \frac{E}{\sigma^2}\,,
\end{equation}
\begin{equation}\label{eq:elliptichessiu}
\|\nabla^2_{\!x}\,u_\sigma\|_{L^2(\Omega)}\leq C(\widehat\Omega)\, \frac{E}{\sigma^2}\,,
\end{equation}
\begin{equation}\label{eq:ellipticgradu}
\|\nabla_{\!x}\, u_\sigma\|_{L^2(\Omega)}
\leq \frac{E}{\sigma\sqrt{1+\mu_1(\Omega)\sigma^2}}\,,
\end{equation}
\begin{equation} \label{eq:ellipticgradu2}
\|\nabla_{\!x}\, u_\sigma\|_{L^2(\Omega)} \leq \frac{1}{2\sigma}  \|\eta_1\|_{L^2(\Omega)}
\end{equation}
and
\begin{equation}\label{eq:elliptictracegradu}
\|\nabla_{\!x}\, u_\sigma\|_{L^2(\partial\Omega)}
\leq C(\widehat\Omega) \, \frac{E}{\sigma^{\frac32}}\,,
\end{equation}
If in addition $\eta_1$ is zero-averaged, then 
\begin{equation} \label{eq:ellipticetau}
0\leq \langle \eta_1,u_\sigma \rangle_{L^2(\Omega)} \leq \frac{\|\eta_1\|_{L^2(\Omega)}^2}{1+\mu_1(\Omega)\sigma^2}\,.
\end{equation}
\end{lem}

\begin{proof}
Testing the equation against $u_\sigma$, we have
\begin{equation}\label{equ_test1}
\|u_\sigma\|_{L^2(\Omega)}^2+\sigma^2 \|\nabla_{\!x}\, u_\sigma\|^2_{L^2(\Omega)}
= \langle \eta_1,u_\sigma\rangle_{L^2(\Omega)}\,.
\end{equation}
This immediately gives the nonnegativity part of~\cref{eq:ellipticetau}.
Using the fact that we have $\sigma^2 \Delta_x u_\sigma = u_\sigma-\eta_1\in L^2(\Omega)$ (where the identity is written in the distributional sense) and~\eqref{equ_test1}, we have 
\begin{equation}\label{equ_test2}
E^2 = \|\eta_1-u_\sigma\|^2_{L^2(\Omega)} + \langle \eta_1-u_\sigma,u_\sigma \rangle_{L^2(\Omega)}
= \sigma^4 \|\Delta_x u_\sigma\|^2_{L^2(\Omega)} + \sigma^2 \|\nabla_{\!x}\, u_\sigma\|^2_{L^2(\Omega)}\,.
\end{equation}
Observe that the non-negativity of $\langle \eta_1,u_\sigma\rangle_{L^2(\Omega)}$ implies
\[
E^2\leq \|\eta_1\|_{L^2(\Omega)}^2\,.
\]
From \cref{equ_test2}, we also immediately deduce \cref{eq:ellipticlapu}.
By \cref{prop:hessbylap}, since $u_\sigma\in H^2(\Omega)$ with Neumann boundary
condition and $\Omega$ has a $\CC^2$ boundary, we have
\[
\| \nabla_{\!x}^2\, u_\sigma\|_{L^2(\Omega)} \leq  C(\widehat\Omega)\|\Delta_x u_\sigma\|_{L^2(\Omega)}\,,
\]
which gives \cref{eq:elliptichessiu} in view of~\cref{eq:ellipticlapu}.
Notice that the above is true with $C(\widehat\Omega)=1$ if $\Omega$ is convex.
Next, observe that \cref{equ_test1} rewrites
\[
\sigma^2 \|\nabla_{\!x}\, u_\sigma\|^2_{L^2(\Omega)}
= \langle \eta_1-u_\sigma,u_\sigma\rangle_{L^2(\Omega)}
= \sigma^2 \langle -\Delta_x u_\sigma,u_\sigma\rangle_{L^2(\Omega)}\,,
\]
so that, still denoting by $\langle w \rangle$ the average of $w$ over $\Omega$, using the fact that
$\langle \Delta_x u_\sigma\rangle=0$ and the Poincaré--Wirtinger inequality, we find
\[
\|\nabla_{\!x}\, u_\sigma\|^2_{L^2(\Omega)}
= \langle -\Delta_x u_\sigma,u_\sigma-\langle u_\sigma\rangle\,\rangle_{L^2(\Omega)}
\leq \mu_1(\Omega)^{-\frac12} \|\Delta_x u_\sigma\|_{L^2(\Omega)} \|\nabla_{\!x}\, u_\sigma\|_{L^2(\Omega)}\,.
\]
Inserting the implied inequality
\[
\mu_1(\Omega) \|\nabla_{\!x}\, u_\sigma\|^2_{L^2(\Omega)}
\leq \|\Delta_x u_\sigma\|_{L^2(\Omega)}^2
\]
into \cref{equ_test2} yields \cref{eq:ellipticgradu}. Inequality~\eqref{eq:ellipticgradu2} directly comes from~\eqref{equ_test1} combined with the Cauchy--Schwarz and Young inequalities.
Then, inserting \cref{eq:elliptichessiu,eq:ellipticgradu} into the multiplicative trace inequality
given by \cref{prop:multtraceineq}, we deduce
\[
\begin{aligned}
	\|\nabla_{\!x}\,u_\sigma\|_{L^2(\partial\Omega)}^2
&\leq C(\widehat\Omega) 
\|\nabla_{\!x}\, u_\sigma\|_{L^2(\Omega)} \left(\d_\Omega^{-1}\|\nabla_{\!x}\, u_\sigma\|_{L^2(\Omega)}+
\|\nabla^2_{\!x}\,u_\sigma\|_{L^2(\Omega)}\right)\\
&\leq C(\widehat \Omega)   \frac{E^2}{\d_\Omega \, \sigma^2(1+\mu_1(\Omega)\sigma^2)} 
+ C(\widehat \Omega) \frac{E}{\sigma \sqrt{1+\mu_1(\Omega)\sigma^2}} \frac{E}{\sigma^2}
\\
&\leq C(\widehat \Omega)   \frac{E^2}{\d_\Omega \, \sigma^2\sqrt{1+\mu_1(\Omega)\sigma^2}} 
+ C(\widehat \Omega) \frac{E^2}{\sigma^3}\,.
\end{aligned}
\]
Recalling that $ \mu_1(\Omega)=\mu_1(\widehat\Omega) \,\d_\Omega^{-2}$, the previous
inequality implies \cref{eq:elliptictracegradu}.
Finally, if in addition $\eta_1$ is zero-averaged, then so is $u_\sigma$, so that applying the Poincaré--Wirtinger inequality this time in \cref{equ_test1}, one gets
\begin{equation*}
	(1+\mu_1(\Omega)\sigma^2) \|u_\sigma\|^2_{L^2(\Omega)}\leq \langle \eta_1,u_\sigma
	\rangle_{L^2(\Omega)}\,.
\end{equation*}
By using the Cauchy--Schwarz inequality, this implies
\[
\|u_\sigma\|_{L^2(\Omega)}\leq \frac{\|\eta_1\|_{L^2(\Omega)}}{1+\mu_1(\Omega)\sigma^2}\,.
\]
We deduce the upper bound of~\cref{eq:ellipticetau} using this estimate and once more the Cauchy--Schwarz inequality.
\end{proof}

We will also need the following estimates on the derivatives of $u_\sigma$ with respect to the parameter $\sigma$.

\begin{lem}\label{prp:ellipticdsigma}
For $\eta_1\in L^2(\Omega)$ and $\eta_2=0$, let $u_\sigma$ be the solution of~\cref{eq:var}.
Then we have
\begin{equation}\label{eq:ellipticdiffnabla}
\|\nabla_{\!x}\,\partial_\sigma u_\sigma\|_{L^2(\Omega)}
\leq \frac{2E}{\sigma^2}\,,
\end{equation}
\begin{equation}\label{eq:ellipticdiffhess}
\|\nabla_{\!x}^2\,\partial_\sigma u_\sigma\|_{L^2(\Omega)}
\leq C(\widehat\Omega) \frac{E}{\sigma^3}
\end{equation}
and
\begin{equation} \label{eq:ellipticdiffd2sigma}
\|\nabla_{\!x}\,\partial^2_\sigma u_\sigma\|_{L^2(\Omega)}
\leq  \frac{5E}{\sigma^3}\,,
\end{equation}
where we recall that $E=\sqrt{\langle \eta_1-u_\sigma,\eta_1\rangle}$. 
\end{lem}

\begin{proof}
Differentiating the equation satisfied by $u_\sigma$ in $\sigma$, we get
\begin{equation} \label{eq:dsigmausigma}
\partial_\sigma u_\sigma = (\id-\sigma^2 \Delta_x)^{-1}(2\sigma\Delta_x u_\sigma),
\end{equation}
so that \cref{eq:ellipticdiffnabla,eq:ellipticdiffhess} immediately follow from
\cref{prp:ellipticeta1}. Differentiating the equation~\eqref{eq:dsigmausigma}, we obtain 
\begin{equation} \label{eq:d2sigmausigma}
\partial^2_\sigma u_\sigma = (\id - \sigma^2 \Delta_x)^{-1} (4\sigma \Delta_x \partial_\sigma u_\sigma + 2 \Delta_x u_\sigma)\,.
\end{equation}
Since by \cref{eq:ellipticlapu} and \cref{eq:dsigmausigma} one has
\[
\|\Delta_x\partial_\sigma u_\sigma\|_{L^2(\Omega)}\leq \frac{2E}{\sigma^3}\,,
\]
we can then deduce that~\eqref{eq:ellipticdiffd2sigma} holds using~\cref{eq:d2sigmausigma}, \cref{eq:ellipticgradu2} and~\cref{eq:ellipticlapu}.
\end{proof}

\subsection{Coercivity estimate} \label{subsec:coercivity}
 If $L=L_1$ is given by~\eqref{def:L1}, we have the following coercivity property for $L$ in~$\HH$ (see~\cite{DGP,Mellet-Indiana}):	$	\langle L g, g \rangle_{\HH} \leq - \zeta_0 \|g^\perp\|^2_\HH	$ which reduces to the equality~$	\langle L g, g \rangle_{\HH} = - \|g^\perp\|^2_{\HH}$ 
when $\zeta \equiv 1$, i.e., for the operator~\eqref{def:BGK}. In the case of the L\'evy--Fokker--Planck operator introduced in~\eqref{def:LFP}, we know from~\cite{Gentil-Imbert} that there exists a constant $C>0$ such that $\langle L g, g \rangle_{\HH} \leq -C\nu_0\|g^\perp\|^2_{\HH}$ where $\nu_0>0$ has been defined in~\eqref{eq:bound-nu}. 
 In short, for all the collision operators that we consider, there exists $\lambda_0>0$ such that 
 	\begin{equation} \label{eq:coercivity}
	\langle L g, g \rangle_{\HH} \leq - \lambda_0 \|g^\perp\|^2_{\HH} \,.
	\end{equation}
	The same arguments as the ones used in~\cite[Lemma~3.1]{BCMT} give the following result (a crucial point being that the distribution $M$ is radially symmetric). We provide a proof of this result for the sake of completeness. 
\begin{lem} \label{lem:coercivity} 
For any distribution~$g\in \mathcal{D}$, we have
	\begin{equation*} \label{eq:coercivity2}
	\langle \Lambda^\eps g, g \rangle_{\HH} \leq - \frac{\lambda_0}{\eps^{2s}} \|g^\perp\|^2_\HH - \frac{\eps^{1-2s}}{2} \Big\|\sqrt{\alpha(2-\alpha)} D^\perp g_+\Big\|^2_{\partial \HH_+}\,. 
	\end{equation*}
\end{lem}

\begin{proof}
By definition of $\Lambda^\eps$, we have
	\[
	\langle \Lambda^\eps g, g \rangle_{\HH}
	= \eps^{-2s} \langle Lg,g\rangle_\HH - \eps^{1-2s} \langle v \cdot \nabla_{\!x}\,g,g \rangle_\HH \,.
	\]
The first term is estimated thanks to~\eqref{eq:coercivity}. Concerning the second one, an integration by parts gives 
	\[
	\la  v \cdot \nabla_{\!x}\,g , g \ra_{\HH}  
	= \int_{\Omega \times \R^d} (v \cdot \nabla_{\!x}\,g) \,g\, M^{-1} \, \d x \, \d v 
	= \frac12 \int_{\Sigma} \gamma g^2 M^{-1} n(x) \cdot v \, \d\mathscr{H}^{d-1}(x) \, \d v\,.
	\]
We then decompose $\gamma g^2 =  g_{+}^2 \mathbf 1_{\Sigma_{+}} + g_{-}^2 \mathbf 1_{\Sigma_{-}}$ and use the boundary condition $\gamma_- g = \RR \gamma_+g$ with $\RR$ given in~\eqref{eq:boundary} to obtain
        \begin{align*}
        &\la  v \cdot \nabla_{\!x}\,g , g \ra_{\HH} \\
        &\qquad
        = \frac12 \int_{\Sigma_{+}}  g_{+}^2 M^{-1} |n(x) \cdot v| \, \d\mathscr{H}^{d-1}(x) \, \d v 
        - \frac12 \int_{\Sigma_{-}}  g_{-}^2 M^{-1} |n(x) \cdot v| \, \d\mathscr{H}^{d-1}(x) \, \d v \\
        &\qquad
        = \frac12 \int_{\Sigma_{+}}  g_{+}^2 M^{-1} |n(x) \cdot v |\, \d\mathscr{H}^{d-1}(x) \, \d v \\
        &\qquad\quad
        - \frac12 \int_{\Sigma_{-}}  \big \{ (1-\alpha(x))g_{+}(x,R_x v) + \alpha(x) D g_{+}(x,v)  \big\}^2 M^{-1} |n(x) \cdot v| \, \d\mathscr{H}^{d-1}(x) \, \d v \, .
        \end{align*}
We apply the change of variables $v \to R_x v$, so that $\Sigma_{-}$ transforms into $\Sigma_{+}$, which yields
        \begin{align*}
        \la  v \cdot \nabla_{\!x}\,g , g \ra_{\HH}  
        &= \frac12 \int_{\Sigma_{+}}  g_{+}^2 M^{-1} |n(x) \cdot v| \, \d\mathscr{H}^{d-1}(x) \, \d v \\
        &\quad
        - \frac12 \int_{\Sigma_{+}} \big \{ (1-\alpha(x))g_{+} + \alpha(x) D g_{+}  \big \}^2 M^{-1} |n(x) \cdot v |\, \d\mathscr{H}^{d-1}(x) \, \d v\,,
        \end{align*}
since $Dg_{+}(x,R_x v) = Dg_{+}(x,v)$ because $M$ is radially symmetric and $|n(x) \cdot R_x v| = |n(x) \cdot v|$.
Writing $g_{+} = D^\perp g_{+} + D g_{+}$, one has 
	\begin{multline*}
	\int_{\Sigma_+} g_+^2 M^{-1}  n(x) \cdot v \, \d\mathscr{H}^{d-1}(x) \, \d v 
	= \int_{\Sigma_+} (Dg_+)^2 M^{-1}  n(x) \cdot v \, \d\mathscr{H}^{d-1}(x) \, \d v \\
	+\int_{\Sigma_+} (D^\perp g_+)^2 M^{-1}  n(x) \cdot v \, \d\mathscr{H}^{d-1}(x) \, \d v,
	\end{multline*}
since $Dg_+(x,\cdot) \perp D^\perp g_+(x,\cdot) $ in $L^2(\Sigma^x_+; M^{-1}(v)n(x)\cdot v\,\d v)$ as recalled in \cref{subsec:notation}. 
Altogether, we get 
	\begin{align*}
	&\la  v \cdot \nabla_{\!x}\,g , g \ra_{\HH}  \\
	&\enspace
	= \frac12 \int_{\Sigma_{+}} \hspace{-1pt} \left\{ (Dg_{+})^2 +  (D^\perp g_{+})^2 - [ (1-\alpha(x)) D^\perp g_{+} +  D g_{+} ]^2\right\} M^{-1} n(x) \cdot v \, \d\mathscr{H}^{d-1}(x) \, \d v \\
	&\enspace
	= \frac12 \int_{\Sigma_{+}} \hspace{-1pt} \left\{ [1- (1-\alpha(x))^2] (D^\perp g_{+})^2 - 2(1-\alpha(x))  Dg_+ D^\perp g_+  \right\} M^{-1} n(x) \cdot v \, \d\mathscr{H}^{d-1}(x) \, \d v \\
	&\enspace
	= \frac12 \int_{\Sigma_{+}}\hspace{-1pt} \alpha(x)(2-\alpha(x)) (D^\perp g_{+})^2M^{-1} n(x) \cdot v \, \d\mathscr{H}^{d-1}(x) \, \d v\,,
	\end{align*}
which provides the wanted result. 
\end{proof}

\subsection{Boundary terms} 
We here recall a result coming from~\cite[Lemma~3.2]{BCMT} that shall be useful to deal with the boundary terms. We do not write the proof of this result here as it is in the same spirit as the one of~\cref{lem:coercivity} and exactly follows the lines of the one of~\cite[Lemma~3.2]{BCMT}. 
\begin{lem}\label{lem:boundary}
Consider $g\in\mathcal{D}$, which satisfies $\gamma_- g = \RR \gamma_+ g$ where $\RR$ is defined in~\eqref{def:reflection}. For $\mathscr{H}^{d-1}$-a.e. $x \in \partial\Omega$ and any suitable test function $\phi: \R^d \to \R$, there holds
	\begin{align*}
	\int_{\R^d}  \phi(v) \gamma g(x,v) \, n(x) \cdot v \, \d v  
	& = \int_{\Sigma^x_{+}}  \phi(v) \alpha(x) D^\perp g_{+} \, n(x) \cdot v \, \d v  \\
	&\quad 
	+ \int_{\Sigma^x_{+}}  \left\{ \phi(v) - \phi(R_x v) \right\}  (1-\alpha(x)) D^\perp g_{+} \, n(x) \cdot v \, \d v  \\
	&\quad 
	+ \int_{\Sigma^x_{+}}   \left\{ \phi(v) - \phi(R_x v) \right\}  D g_{+} \, n(x) \cdot v \, \d v\, .
	\end{align*}
\end{lem}

\subsection{New macroscopic quantities} \label{subsec:macro}

When integrating our kinetic equation in velocity, we see that higher order moments in velocity appear due to the presence of the transport operator. It is thus natural to introduce, in addition to the local mass $\rho[g]$ already defined in~\eqref{def:BGK}, the moment of order one defined through:
	\[
	j[g] (x)\coloneqq \int_{\R^d} g (x,v) v \, \d v \,.
	\]
A distinctive feature of our framework is that we consider heavy-tailed distributions. In view of that, we introduce new ``twisted'' macroscopic quantities
	\begin{equation*}
	R [g] (x) \coloneqq \int_{\R^d} g (x,v) \, \frac{\d v}{\langle v \rangle^2} 
	\qquad \text{and} \qquad 
	J [g] (x)\coloneqq \int_{\R^d} g (x,v)\, v \,\frac{\d v}{\langle v \rangle^2} \,. 
	\end{equation*}
Observe that the additional weight of order $-2$  counterbalances the lack of integrability of our heavy-tailed distributions. In particular $R[v_iv_jg]$, with $1\leq i,j\leq d$, which appears in the momentum equation, is well defined for a merely integrable $g$. 
Let us derive some preliminary estimates on the twisted macroscopic quantities.
\begin{lem} \label{lem:macro}
For $g\in\mathcal{D}$, the macroscopic quantities $R [g]$ and $J [g]$ satisfy the following identities:
    \begin{align} 
	R [g](x)& = c_{-2} \rho[g](x) + \int_{\R^d} g^\perp(x,v) \, \frac{\d v}{\langle  v
	\rangle^2}\,,\quad\text{ with  }c_{-2} \coloneqq \int_{\R^d} M(v)  \, \frac{\d v}{\langle  v
	\rangle^2}\,, \label{eq:tilderho}\\
	J[g] &= J[g^\perp]\,.\label{eq:tildej}
 	\end{align}
They additionally satisfy the following estimates:
	\begin{align}
	\left\|R[g]\right\|_{L^2_x(\Omega)} &\leq \left\|g\right\|_{\HH}\,,\label{eq:tilderho3}\\  
 	\left\|J[g]\right\|_{L^2_x(\Omega)} &\leq \|g^\perp\|_\HH\,.\label{eq:tildej2}\\
 	\end{align}
Finally, because $g$ satisfies $\gamma_- g = \RR \gamma_+g$ with $\RR$ given in~\eqref{eq:boundary}, then 
\begin{equation}
\left\|J[g]\cdot n\right\|_{L^2_x(\partial\Omega)}
 	\leq \left\|\sqrt{\alpha(2-\alpha)} D^\perp g_+ \right\|_{\partial\HH_+}\label{eq:tildej2boundary}\,.
\end{equation}
\end{lem}
\begin{proof}
Recalling that $g = \rho[g]M + g^\perp$ we immediately have \eqref{eq:tilderho} 
and we deduce~\eqref{eq:tildej} from the symmetry of $M$. Then by the Cauchy--Schwarz inequality,
	\[
	\left\|R[g]\right\|_{L^2_x(\Omega)}  
	\leq \left(\int_{\R^d} \frac{M(v)}{\langle  v \rangle^4} \, \d v\right)^{\frac12} \|g\|_\HH\leq \|g\|_\HH\,,
	\]
which is \eqref{eq:tilderho3}. The estimates \cref{eq:tildej2} are proved in a similar way.
Finally, using \cref{lem:boundary} with $\phi(v)=\langle v\rangle^{-2}$ we have, for $\mathscr{H}^{d-1}$-a.e. $x\in\partial\Omega$,
\[
J[g]\cdot n(x)
= \alpha(x)\int_{\Sigma^x_+} D^\perp g_+ (n(x)\cdot v)\frac{\d v}{\langle v\rangle^{2}}
\]
which gives \cref{eq:tildej2boundary}, since $\alpha\leq \sqrt{\alpha(2-\alpha)}$.
\end{proof}

\section{Fractional hypocoercivity} \label{sec:proof}
\subsection{Definition of the new norm} \label{subsec:newnorm}
The purpose of this section is to prove \cref{theo:main}. 
To this end, we introduce a new bilinear form 
on $\HH$ defined for $\eps >0$ through
	\begin{multline} \label{def:scalar}
	\langle\!\langle f, g \rangle\!\rangle_\eps \coloneqq 
	\langle f, g \rangle_\HH  
	+ \delta \langle \eps v\cdot\nabla_{\!x} (\id-\eps^2\langle v\rangle^2\Delta_x)^{-1} R[f]M, g
	\rangle_{\HH} \\
	+ \delta \langle \eps v\cdot\nabla_{\!x} (\id-\eps^2\langle v\rangle^2\Delta_x)^{-1} R[g]M, f
	\rangle_{\HH}
	\end{multline}
where $\delta>0$ will be chosen small enough and $(\id-\eps^2\langle v\rangle^2\Delta_x)^{-1} R[\cdot]$ is defined thanks to results recalled in \cref{subsec:elliptic} with $\eta_1 =  R[\cdot] $, $\eta_2=0$ and $\sigma = \eps \langle v \rangle$. 
We denote by $\Nt\cdot \Nt_\eps$ the associated norm which writes as follows
	\begin{equation} \label{def:Nt}
	\Nt g \Nt^2_\eps \coloneqq 
	\|g\|^2_\HH + 2\delta \langle \eps v\cdot\nabla_{\!x} (\id-\eps^2\langle v\rangle^2\Delta_x)^{-1}
	R[g]M, g \rangle_{\HH}\,.
	\end{equation}
Notice that by \cref{eq:ellipticgradu2} and \cref{eq:tilderho3}, we have for any $g\in\HH$,
\begin{equation}\label{eq:conselliptic}
\begin{aligned}
\|\eps v\cdot\nabla_{\!x} (\id-\eps^2\langle v\rangle^2\Delta_x)^{-1} R[g]M\|_{\HH}^2
&\leq \frac{1}{4}\int_{\R^d} \|R[g]\|_{L^2_x(\Omega)}^2\frac{|v|^2}{\langle v\rangle^2} M(v)\, \d v \\
&\leq \frac{1}{4} \left(\int_{\R^d} \frac{|v|^2}{\langle v \rangle^2} M(v) \, \d v \,\right) \|g\|_{\HH}^2\,.
\end{aligned}
\end{equation}

The fact that $\int_{\R^d} \frac{|v|^2}{\langle v \rangle^2} M(v) \, \d v <1$ allows us to obtain the following lemma:

\medskip
\begin{lem} \label{lem:eqnorm}
The norm $\Nt \cdot \Nt_\eps$ defined in~\eqref{def:Nt} and the classical norm $\|\cdot \|_\HH$ on
$\HH$ are equivalent for $\delta \leq 1$ uniformly in $\eps>0$. 
\end{lem}

From now on, we choose $\delta \in (0,1]$ so that the conclusion of~\cref{lem:eqnorm} holds; it will later be taken small enough at the end of~\cref{subsec:proof} to complete the proof of~\cref{theo:main}.

\subsection{Proof of  \texorpdfstring{\cref{theo:main}}{Theorem~\ref{theo:main}}} \label{subsec:proof}
The goal is now to prove \cref{theo:main}. For this, we consider a distribution $f \in \mathrm{Dom}(\Lambda^\eps)$ satisfying the boundary condition in~\eqref{eq:dtf=Lf} and with vanishing global mass, i.e., $\int_{\Omega \times \R^d} f(x,v) \, \d v \, \d x =0$. The first term in $\la\!\la \Lambda^\eps f, f \ra\!\ra_\eps$ is going to be estimated thanks to \cref{lem:coercivity}. It now remains to estimate the two cross terms appearing in~$\la\!\la \Lambda^\eps f, f \ra\!\ra_\eps$ (see the definition~\eqref{def:Nt}). This is the purpose of the  following lemma and proposition. 

In this section, as in the statement of~\cref{theo:main}, the multiplicative constants in $\lesssim$ depend only on $\eps_0$, $\widehat\Omega$, $s$, $d$, $\zeta$, $\nu$ (defined in \eqref{def:L1} and~\eqref{def:LFP}) and the coercivity constant~$\lambda_0$.

\begin{lem} \label{lem:A}
For any $f \in \mathcal{D}$, there holds
	\begin{multline} \label{eq:A}
	\langle  \eps v\cdot \nabla_{\!x} (\id-\eps^2\langle v\rangle^2\Delta_x)^{-1} R[\Lambda^\eps f]M, f \rangle_{\HH}\\ 
	\lesssim \frac{1}{\eps^{2s}}\|f^\perp\|_\HH^2 +\frac{1}{\eps^{2s-1}} \left\|\sqrt{\alpha(2-\alpha)} D^\perp f_+\right\|_{\partial\HH_+}^2\,.
	\end{multline}
\end{lem}
\begin{proof} 
	First observe that $v \mapsto (\id - \eps^2 \langle v \rangle^2 \Delta_x)^{-1} R[\Lambda^\eps f]$ and $M$ depend only on $|v|$ so that
	\[
	\langle \eps v\cdot \nabla_{\!x} (\id-\eps^2\langle v\rangle^2\Delta_x)^{-1}  R[\Lambda^\eps
	f]M, f \rangle_{\HH}  = \langle \eps v\cdot \nabla_{\!x} (\id-\eps^2\langle
	v\rangle^2\Delta_x)^{-1}  R[\Lambda^\eps f]M, f^\perp \rangle_{\HH}\,.
	\]
Then we notice that
	\[
	R[\Lambda^\eps f]
	= - \frac{1}{\eps^{2s-1}} \operatorname{div}_x J[f]
	+ \frac{1}{\eps^{2s}}R[L f] \,. 
	\]
It allows us to write
	\[
	\langle \eps v\cdot \nabla_{\!x} (\id-\eps^2\langle v\rangle^2\Delta_x)^{-1}  R[\Lambda^\eps f]M, f^\perp \rangle_{\HH} = A_1+A_2\,,
	\]
with 
    \begin{equation*}
    \left\{
	\begin{aligned}
	A_1&\coloneqq -\frac{1}{\eps^{2s-1}} \left\langle \eps v\cdot \nabla_{\!x} (\id-\eps^2\langle
	v\rangle^2\Delta_x)^{-1}  \operatorname{div}_x J[f]M, f^\perp \right\rangle_{\HH}\,,\\
	A_2 &\coloneqq \frac{1}{\eps^{2s}}\left\langle \eps v\cdot \nabla_{\!x} (\id-\eps^2\langle
	v\rangle^2\Delta_x)^{-1}  R[L f]M,  f^\perp \right\rangle_{\HH}\,. 
	\end{aligned}
	\right.
	\end{equation*}
Thanks to the Cauchy--Schwarz inequality, \cref{eq:ellipticeta2}, \cref{eq:tildej2} and \cref{eq:tildej2boundary} we have 
	\begin{align*}
	A_1 &\leq \frac{1}{\eps^{2s-1}} \left\|\eps v\cdot \nabla_{\!x} (\id-\eps^2\langle
	v\rangle^2\Delta_x)^{-1} \operatorname{div}_x J[f]M\right\|_{\HH} \|f^\perp\|_{\HH} \\
	&\lesssim \left(\frac{1}{\eps^{2s}}  \left\|J[f]\right\|_{L^2_x(\Omega)}+\frac{1}{\eps^{2s-\frac{1}{2}}}  \left\|J[f]\cdot n\right\|_{L^2_x(\partial\Omega)}\right)\|f^\perp\|_{\HH}\\
	&\lesssim 
	 \frac{1}{\eps^{2s}}  \|f^\perp\|_{\HH}^2+\frac{1}{\eps^{2s-\frac{1}{2}}}  \left\|\sqrt{\alpha(2-\alpha)} D^\perp f_+\right\|_{\partial\HH_+}\|f^\perp\|_{\HH}\\
	&\lesssim \frac{1}{\eps^{2s}}  \|f^\perp\|_\HH^2
	+\frac{1}{\eps^{2s-1}} \left\|\sqrt{\alpha(2-\alpha)} D^\perp f_+\right\|_{\partial\HH_+}^2
	\,.
	\end{align*}
Here we also used Young's inequality for the last inequality.
In a similar fashion, using~\cref{eq:conselliptic}, we have
	\[
	A_2 \leq \frac{1}{2\eps^{2s}}\|R[Lf]\|_{L^2_x(\Omega)} \|f^\perp\|_\HH \,. 
	\]
We then need to examine in turn the cases $L=L_1$ and $L=L_2$. For $L=L_1$, we first remark that 
	\begin{align*}
	R[Lf] (x)
	&= R [L f^\perp] (x) \\
	&=- \int_{\R^d \times \R^d} \zeta(x,v,v') \left[f^\perp(x,v') M(v) - f^\perp (x,v) M(v') \right] \, \d v' \, \frac{|v|^2}{\langle v \rangle^2} \, \d v 
	\end{align*}
where we used the fact that $\int_{\R^d} Lf^\perp \, \d v= 0$ since $\zeta$ is symmetric in $(v,v')$. From this and~\eqref{eq:bound-sigma}, we deduce that 
	\begin{align*}
	\|R[Lf]\|_{L^2_x(\Omega)} 
	&\lesssim \left\| \int_{\R^d} |f^\perp| \, \d v \right\|_{L^2_x(\Omega)} \int_{\R^d} M(v) \frac{|v|^2}{\langle v \rangle^2 } \, \d v 
	+ \left\| \int_{\R^d} |f^\perp| \frac{|v|^2}{\langle v \rangle^2} \, \d v\right\|_{L^2_x(\Omega)} \\
	&\leq \|f^\perp\|_\HH \int_{\R^d} M(v) \frac{|v|^2}{\langle v \rangle^2 } \, \d v 
	+ \|f^\perp\|_\HH \left(\int_{\R^d} M(v) \frac{|v|^4}{\langle v \rangle^4 } \, \d v \right)^{\frac12}\,, 
	\end{align*}
where we used the Cauchy--Schwarz inequality in velocity to get the last estimate. Thus 
	\begin{equation}\label{estimL}
	\|R[Lf]\|_{L^2_x(\Omega)} \lesssim  \|f^\perp\|_\HH\,. 
	\end{equation}
In the end, we obtain that 
	\begin{equation} \label{eq:A2}
	A_2 \lesssim \frac{1}{\eps^{2s}} \|f^\perp\|^2_\HH\,. 
	\end{equation}
For $L=L_2$, one has
	\[
	R[Lf] (x)
	= R [L f^\perp] (x) 
	=-\nu(x)\int_{\R^d} \left(\nabla_{\!v}\cdot(vf^\perp) -(-\Delta_v)^{s}f^\perp\right)\frac{|v|^2}{\langle v\rangle^{2}}\,\d v\,.
	\]
Using the adjoint operator one obtains 
	\begin{align*}
	\|R[Lf]\|_{L^2_x(\Omega)}^2&
	=\int_\Omega\nu^2(x)\left(\int_{\R^d} f^\perp\left( g(v) + h(v)\right)\, \d v \right)^2 \,\d x
	\end{align*}
where
	\[
	g(w) \coloneqq w\cdot \nabla_w\left(\frac{|w|^2}{\langle w\rangle^2}\right)
	\quad\text{ and }\quad
	h(w) \coloneqq (-\Delta_w)^s\left(\frac{|w|^2}{\langle w\rangle^2}\right).
	\]
Notice that $\langle w\rangle^{-2} = 1-\frac{|w|^2}{\langle w\rangle^{2}}$, thus we have
	\[
	|g(w)|\lesssim \frac{|w|^2}{1+|w|^4}
	\]
and
	\begin{equation*}
    \begin{aligned}
	\left|h(w)\right|
	&\lesssim \int_{\R^d} \frac{\left|2\langle w\rangle^{-2} -\langle w+\bar w\rangle^{-2}-\langle w-\bar w\rangle^{-2}\right|}{|\bar w|^{d+2s}}\,\d \bar w\\
	&\lesssim \int_{|\bar w|\leq1} \frac{|\bar w|^2}{|\bar w|^{d+2s}}\,\d \bar w + \int_{|\bar w|>1} \frac{1}{|\bar w|^{d+2s}}\,\d \bar w\lesssim 1\,.
	\end{aligned}\end{equation*}
Therefore, we obtain \eqref{estimL} for  $L=L_2$ as well as \eqref{eq:A2}.
We conclude the proof by gathering the estimates on $A_1$ and $A_2$.
\end{proof}

\begin{prop} \label{lem:B}
There exist positive constants $\kappa,C$ such that for any $\eps>0$ and $f \in \mathcal{D}$, there holds
	\begin{multline} \label{eq:B}
	\langle  \eps v\cdot \nabla_{\!x} (\id-\eps^2\langle v\rangle^2\Delta_x)^{-1} R[f] M, \Lambda^\eps f \rangle_{\HH}  \\
	\leq
	- \kappa G[f] + \frac{C}{\eps^{2s}} \|f^\perp\|^2_\HH + \frac{C}{\eps^{2s-1}}\left\|\sqrt{\alpha(2-\alpha)} D^\perp f_+\right\|_{\partial \HH_+}^2
	\end{multline}
where the gain term is defined by
\begin{equation}\label{eq:defgain}
G[f]\coloneqq \frac{1}{\eps^{2s}}\int_{\R^d} \langle \rho[f]-u[f],\rho[f]\rangle_{L^2_x(\Omega)}\,\frac{|v|^2}{\langle
v\rangle^2} M(v)\, \d v
\end{equation}
with $u[f] \coloneqq (\id-\eps^2\langle v\rangle^2\Delta_x)^{-1} \rho[f]$ introduced in \cref{subsec:elliptic}.
In addition, if we fix~$\eps_0>0$, then for any $\eps \in (0,\eps_0\d_\Omega)$, we have
\[
G[f]\gtrsim \d_\Omega^{-2s}\|\rho[f]\|_{L^2_x(\Omega)}^2\,.
\]
\end{prop}
\begin{proof}
We first write
	\[
\langle  \eps v\cdot \nabla_{\!x} (\id-\eps^2\langle v\rangle^2\Delta_x)^{-1} R[f] M, \Lambda^\eps f \rangle_{\HH}= B_1+B_2
	\]
with 
\begin{equation*}
	\left\{\begin{aligned} 
	B_1 &\coloneqq - \frac{1}{\eps^{2s-1}} \left\langle \eps v\cdot \nabla_{\!x} (\id-\eps^2\langle v\rangle^2\Delta_x)^{-1}
	R [f]M, v\cdot\nabla_{\!x}\, f \right\rangle_{\HH}\\
	B_2 &\coloneqq \frac{1}{\eps^{2s}} \left\langle \eps v\cdot \nabla_{\!x} (\id-\eps^2\langle v\rangle^2\Delta_x)^{-1}
	R[f]M,L f \right\rangle_{\HH}\,. 
	\end{aligned}\right.
\end{equation*}
We first estimate $B_1$. From~\eqref{eq:tilderho}, we have
  \begin{align*}
    B_1 &= - \frac{c_{-2}}{\eps^{2s-1}} \left\langle \eps v_i\partial_{x_i}(\id-\eps^2\langle v\rangle^2\Delta_x)^{-1}
     \rho[f]M , v_j\partial_{x_j} f \right\rangle_{\HH} \\
     & \quad \qquad
- \frac{1}{\eps^{2s-1}} \left\langle \eps v_i\partial_{x_i}(\id-\eps^2\langle v\rangle^2\Delta_x)^{-1}
     R[f^\perp]M, v_j\partial_{x_j} f \right\rangle_{\HH} \\
     &
     \eqqcolon B_{11} + B_{12}
	\end{align*}
where we used the summation convention over repeated indices.
Thanks to an integration by parts and the micro-macro decomposition of $f$, we split $B_{11}$ into
\[
\begin{aligned}
B_{11}
&= \frac{c_{-2}}{\eps^{2s-1}} \left\langle \eps   v_j v_i\partial_{x_j}\partial_{x_i}(\id-\eps^2\langle v\rangle^2\Delta_x)^{-1}
     \rho[f]M ,\rho[f]M \right\rangle_{\HH}\\
&\qquad + \frac{c_{-2}}{\eps^{2s-1}} \left\langle \eps v_j v_i\partial_{x_j}\partial_{x_i}(\id-\eps^2\langle v\rangle^2\Delta_x)^{-1}
     \rho[f]M , f^\perp \right\rangle_{\HH}\\
&\qquad - \frac{c_{-2}}{\eps^{2s-1}} \int_{\Sigma} (v\cdot n(x))\left[\eps v\cdot \nabla_{x}(\id-\eps^2\langle
v\rangle^2\Delta_x)^{-1} \rho[f]\right] (\gamma f)\, \d v\,\d\mathscr{H}^{d-1}(x)\\
&\eqqcolon B_{111}+B_{112}+B_{113}\,.
 \end{aligned}
\]
Using that $v \mapsto (\id - \eps^2 \langle v \rangle^2 \Delta_x)^{-1} \rho[f]$ depends only on $|v|$ and that $M$ is radial, we have
\begin{equation}\label{eq:estimB111}
\begin{aligned}
B_{111}
&= \frac{c_{-2}}{d\eps^{2s}} \left\langle \eps^2   |v|^2\Delta_x(\id-\eps^2\langle
v\rangle^2\Delta_x)^{-1} \rho[f]M ,\rho[f]M \right\rangle_{\HH}\\
&= \frac{c_{-2}}{d\eps^{2s}} \left\langle \frac{|v|^2}{\langle v\rangle^2}
(u[f]-\rho[f])M ,\rho[f]M \right\rangle_{\HH}\\
&= -2\kappa G[f]
\end{aligned}
\end{equation}
with $G[f]$ defined in \cref{eq:defgain} and $\kappa = \frac{c_{-2}}{2d}$.
Notice that $G[f]$ is non-negative by \cref{prp:ellipticeta1}.
Next, by \cref{eq:elliptichessiu} we estimate
\begin{equation}\label{eq:estimB112}
\begin{aligned}
B_{112}
&\lesssim \frac{1}{\eps^{2s}} \left(\int_{\R^d}\|\eps^2\langle v\rangle^2\nabla_{\!x}^2(\id-\eps^2\langle
v\rangle^2\Delta_x)^{-1}\rho[f]\|_{L^2_x(\Omega)}^2 \frac{|v|^2}{\langle v\rangle^2}M(v)\,\d v\right)^{\frac12} \|f^\perp\|_{\HH}\\
&\lesssim \frac{1}{\eps^{s}} \sqrt{G[f]}\|f^\perp\|_{\HH}\,.
\end{aligned}
\end{equation}
To deal with~$B_{113}$, we apply \cref{lem:boundary} with
\[
\phi(v) = \eps v\cdot\nabla_{\!x}\,(\id-\eps^2\langle v\rangle^2 \Delta_x)^{-1}\rho[f]\,,
\]
which is such that $ \phi(v) -
\phi(R_x v) = 0 $ thanks to the Neumann boundary condition  
used to define ~$(\id-\eps^2\langle v\rangle^2\Delta_x)^{-1} \rho[f]$ in \cref{subsec:elliptic}.
It gives 
\[
B_{113} =-  \frac{c_{-2}}{\eps^{2s-1}}  \int_{\Sigma_+} \alpha(x) \eps v\cdot \nabla_{\!x}
(\id-\eps^2\langle v\rangle^2\Delta_x)^{-1} \rho[f]\, D^\perp f_+ \, (n(x) \cdot v) \,\d v\,\d \mathscr{H}^{d-1}(x)\,.
\] 
Using the Cauchy--Schwarz inequality in velocity and \cref{eq:elliptictracegradu}, we estimate
\begin{equation}\label{eq:estimB113}
\begin{aligned}
B_{113}
&\leq \frac{1}{\eps^{2s}} \left(\int_{\Sigma_+} \eps^4 M(v) |v|^3 |\nabla_{\!x}\,(\id-\eps^2\langle
v\rangle^2\Delta_x)^{-1}\rho[f]|^2\,\d v\,\d\mathscr{H}^{d-1}(x)\right)^{\frac12} \|\alpha D^\perp
f_+\|_{\partial\HH_+}\\
&\lesssim \frac{1}{\eps^{2s-\frac12}} \left(\int_{\R^d} M(v) \frac{|v|^3}{\langle
v\rangle^3}\langle \rho[f]-u[f],\rho[f]\rangle_{L^2_x(\Omega)}\,\d v\right)^{\frac12} \|\alpha D^\perp
f_+\|_{\partial\HH_+}\\
&\lesssim \frac{1}{\eps^{s-\frac12}} \sqrt{G[f]}\left\|\sqrt{\alpha(2-\alpha)} D^\perp
f_+\right\|_{\partial \HH_+}
\end{aligned}
\end{equation}
where we used the fact that $\alpha \leq \sqrt{\alpha(2-\alpha)}$ since $\alpha \in [0,1]$.
Combining \cref{eq:estimB111,,eq:estimB112,,eq:estimB113}, one finds
\begin{equation}\label{eq:estimB11}
B_{11} \leq -2\kappa G[f] +
C\left(\frac{1}{\eps^{s-\frac12}} \left\|\sqrt{\alpha(2-\alpha)} D^\perp
f_+\right\|_{\partial \HH_+} +\frac{1}{\eps^{s}}\|f^\perp\|_{\HH}\right)\sqrt{G[f]}\,.
\end{equation}
We now estimate $B_{12}$. Integrating by parts and using the micro-macro decomposition and the
symmetry of $M$, we have
\[
\begin{aligned}
B_{12}
&= \frac{1}{d\eps^{2s}} \left\langle \eps^2|v|^2\Delta_x (\id-\eps^2\langle v\rangle^2\Delta_x)^{-1}
     R[f^\perp]M ,\rho[f]M \right\rangle_{\HH}\\
&\qquad + \frac{1}{\eps^{2s}} \left\langle \eps^2 v_j
v_i\partial_{x_j}\partial_{x_i}(\id-\eps^2\langle v\rangle^2\Delta_x)^{-1} R[f^\perp]M , f^\perp \right\rangle_{\HH}\\
&\qquad - \frac{1}{\eps^{2s-1}} \int_{\Sigma} (v\cdot n(x))\left[\eps v\cdot \nabla_{x}(\id-\eps^2\langle
v\rangle^2\Delta_x)^{-1} R[f^\perp]\right] (\gamma f)\, \d v\,\d\mathscr{H}^{d-1}(x)\\
&\eqqcolon B_{121}+B_{122}+B_{123}\,.
 \end{aligned}
\]
Since $\eps^2|v|^2\Delta_x (\id-\eps^2\langle v\rangle^2\Delta_x)^{-1}$ is self-adjoint in $L^2_x(\Omega)$, using
\cref{eq:ellipticlapu,eq:tilderho3} we get
\[
B_{121}
= \frac{1}{d\eps^{2s}} \left\langle 
R[f^\perp]M ,\eps^2|v|^2\Delta_x (\id-\eps^2\langle v\rangle^2\Delta_x)^{    -1}\rho[f]M \right\rangle_{\HH}
\lesssim \frac{1}{\eps^s}\|f^\perp\|_{\HH} \sqrt{G[f]}\,.
\]
By \cref{eq:elliptichessiu,eq:tilderho3}, we estimate as well
\[
B_{122} \lesssim \frac{1}{\eps^{2s}} \|f^\perp\|_{\HH}^2\,.
\]
Proceeding as for $B_{113}$ and using \cref{eq:tilderho3}, we find
\[
\begin{aligned}
B_{123}
&\leq \frac{1}{\eps^{2s}} \left(\int_{\Sigma_+} \eps^4 M(v) |v|^3 |\nabla_{\!x}\,(\id-\eps^2\langle
v\rangle^2\Delta_x)^{-1}R[f^\perp]|^2\,\d v\,\d\mathscr{H}^{d-1}(x)\right)^{\frac12} \|\alpha D^\perp
f_+\|_{\partial\HH_+}\\
&\lesssim \frac{1}{\eps^{2s-\frac12}} \left(\int_{\R^d} M(v) \frac{|v|^3}{\langle
v\rangle^3}\|R[f^\perp]\|^2_{L^2_x(\Omega)}\,\d v\right)^{\frac12} \|\alpha D^\perp
f_+\|_{\partial\HH_+}\\
&\lesssim \frac{1}{\eps^{2s-\frac12}} \|f^\perp\|_{\HH}\left\|\sqrt{\alpha(2-\alpha)} D^\perp
f_+\right\|_{\partial \HH_+}.
\end{aligned}
\]
Gathering the previous estimates, we deduce
\begin{equation}\label{eq:estimB12}
B_{12} \lesssim \frac{1}{\eps^s} \|f^\perp\|_{\HH} \sqrt{G[f]}
+\frac{1}{\eps^{2s}} \|f^\perp\|_{\HH}^2
+\frac{1}{\eps^{2s-\frac12}} \|f^\perp\|_{\HH}\left\|\sqrt{\alpha(2-\alpha)} D^\perp
f_+\right\|_{\partial \HH_+}.
\end{equation}
As for $B_2$, we first decompose it into two parts using~\eqref{eq:tilderho}:
\begin{align*}
B_2 
&= \frac{c_{-2}}{\eps^{2s}} \left\langle \eps v\cdot \nabla_{\!x} (\id-\eps^2\langle
v\rangle^2\Delta_x)^{-1} \rho[f]M,L f^\perp \right\rangle_{\HH}\\
&\qquad\qquad
+\frac{1}{\eps^{2s}} \left\langle \eps v\cdot \nabla_{\!x} (\id-\eps^2\langle
v\rangle^2\Delta_x)^{-1} R[f^\perp]M,L f^\perp \right\rangle_{\HH} \\
&\eqqcolon B_{21} + B_{22}\,. 
\end{align*}
When $L=L_1$ (see~\eqref{def:L1}), using~\cref{eq:bound-sigma}, the Cauchy--Schwarz inequality and the fact that $L_1$ is self-adjoint, we get
\[
B_{21}
\lesssim \frac{1}{\eps^{2s}} \|f^\perp\|_{\HH}\left(\int_{\R^d}\int_\Omega |\eps v|^2|\nabla_{\!x}\,(\id-\eps^2\langle v\rangle^2\Delta_x)^{-1}
\rho[f]|^2 M(v)\,\d x\,\d v\right)^{\frac12}.
\]
By \cref{eq:ellipticgradu} and \cref{eq:defgain}, we deduce
\[
B_{21} \lesssim \frac{1}{\eps^{s}}\|f^\perp\|_{\HH}\sqrt{G[f]}\,.
\]
Then, by \cref{eq:ellipticgradu} and \cref{eq:tilderho3} we estimate
\[
B_{22} \lesssim  \frac{1}{\eps^{2s}} \|f^\perp\|_{\HH}^2\,.
\]
When $L=L_2$ (see \cref{def:LFP}), integrating by parts, we have
\[
\begin{aligned}
B_{21}
&= -\frac{c_{-2}}{\eps^{2s}} \int_{\R^d}\int_{\Omega} \nu(x)(-\Delta_v)^s\left[\eps v\cdot
\nabla_{\!x}\,u_{\eps\langle v\rangle}[f]\right]f^\perp \,\d x\,\d v\\
&\qquad\qquad
-\frac{c_{-2}}{\eps^{2s}} \int_{\R^d}\int_{\Omega} \nu(x)v\cdot \nabla_{\!v}\left[\eps v\cdot
\nabla_{\!x}\,u_{\eps\langle v\rangle}[f]\right]f^\perp \,\d x\,\d v\\
&\eqqcolon B_{211}+B_{212}
\end{aligned}
\]
where we introduced the notation $u_{\eps\langle v\rangle}[f]=(\id-\eps^2\langle
v\rangle^2\Delta_x)^{-1}\rho[f]$.
We first estimate $B_{211}$.
By the Cauchy--Schwarz inequality and \cref{eq:bound-nu}, we have
\[
B_{211} \lesssim \frac{1}{\eps^{2s}}\|f^\perp\|_{\HH}
\left(\int_{\R^d}\int_{\Omega}
\left|(-\Delta_v)^s\left[\eps v\cdot\nabla_{\!x}\, u_{\eps\langle v\rangle}[f]\right]\right|^2 M(v)\,\d
x\,\d v\right)^{\frac12}.
\]
Notice that for $W\equiv W(x,v)$, splitting the integral into two parts and
using a Taylor expansion, we can estimate pointwise
\[
\begin{aligned}
&|(-\Delta_v)^s W(x,v)|
\lesssim \left|\int_{\R^d} \frac{W(x,v+h)+W(x,v-h)-2W(x,v)}{|h|^{d+2s}}\,\d h\right|\\
&\quad 
\lesssim \int_{B(0,1)}\int_0^1 \frac{|D^2_v\, W(x,v+\tau h)|}{|h|^{d+2(s-1)}}\,\d\tau\, \d h
+|W(x,v)|+\int_{\R^d} |W(x,v+h)|M(h)\,\d h
\end{aligned}
\]
where $B(0,1)$ is the unit ball in $\R^d$.
Taking $W(x,v) = \eps v\cdot\nabla_{\!x}\, u_{\eps\langle v\rangle}[f]$, we have that
\[
\begin{aligned}
\partial^2_{v_iv_j} W(x,v)
&=  \eps v\cdot \nabla_{\!x}\,\partial^2_{v_iv_j}  u_{\eps\langle v\rangle}[f]
+ \eps\,\partial_{x_i} \partial_{v_j} u_{\eps\langle v\rangle}[f]
+ \eps\,\partial_{x_j} \partial_{v_i} u_{\eps\langle v\rangle}[f]\\
&= \eps^3 (\partial_{v_i}\langle v\rangle) (\partial_{v_j}\langle v\rangle)v\cdot\nabla_{\!x}\,
\partial^2_{\sigma} u_{\eps\langle v\rangle}[f]
+ \eps^2(\partial^2_{v_iv_j}\langle v\rangle) v\cdot\nabla_{\!x}\,\partial_{\sigma}u_{\eps\langle v\rangle}[f]\\
&\qquad\qquad
+\eps^2(\partial_{v_j}\langle v\rangle)\partial_{x_i}\partial_\sigma u_{\eps\langle v\rangle}[f]
+\eps^2(\partial_{v_i}\langle v\rangle)\partial_{x_j}\partial_\sigma u_{\eps\langle v\rangle}[f].
\end{aligned}
\]
By \cref{eq:ellipticdiffnabla} and \cref{eq:ellipticdiffd2sigma}, we observe that
\[
\|D^2_v\, W(\cdot,v+\tau h)\|_{L^2_x(\Omega)}^2
\lesssim \frac{|v+\tau h|^2}{\langle v+\tau h\rangle^2} \langle \rho[f]-u_{\eps\langle v+\tau
h\rangle}[f],
\rho[f]\rangle_{L^2_x(\Omega)}
\]
and by \eqref{eq:ellipticgradu}, that
\[
\|W(\cdot,v)\|_{L^2_x(\Omega)}^2 \lesssim \frac{|v|^2}{\langle v\rangle^2} 
\langle \rho[f]-u_{\eps\langle v\rangle}[f], \rho[f]\rangle_{L^2_x(\Omega)}\,.
\]
Gathering the previous inequalities and using Jensen's inequality, we obtain
\begin{multline*}
\int_{\R^d}\int_{\Omega}
\left|(-\Delta_v)^s\left[\eps v\cdot\nabla_{\!x}\, u_{\eps\langle v\rangle}[f]\right]\right|^2 M(v)\,\d
x\,\d v\\
\lesssim \int_{\R^d}\int_{B(0,1)}\int_0^1 
\frac{|v+\tau h|^2}{\langle v+\tau h\rangle^2|h|^{d+2(s-1)}} \langle \rho[f]-u_{\eps\langle v+\tau
h\rangle}[f], \rho[f]\rangle_{L^2_x(\Omega)} M(v)\,\d \tau\,\d h\,\d v\\
+\int_{\R^d} \frac{|v|^2}{\langle v\rangle^2} \langle \rho[f]-u_{\eps\langle v\rangle}[f], \rho[f]\rangle_{L^2_x(\Omega)} M(v)\,\d v\\
+\int_{\R^d} \int_{\R^d} \frac{|v+h|^2}{\langle v+h\rangle^2} \langle \rho[f]-u_{\eps\langle
v+h\rangle}[f], \rho[f]\rangle_{L^2_x(\Omega)} M(v)M(h)\,\d h\,\d v\,.
\end{multline*}
Since for any $\tau \in (0,1)$,
\[
M(v) \lesssim \frac{M(v+\tau h)}{M(h)}\,,
\]
the first term is bounded from above by the second term after a change of variable in $v$.
For the third term, we need the stronger inequality
\[
M(v)M(h) \lesssim M(v+h)\left(M(v)+M(h)\right)
\]
which follows from $M(v+h)^{-1} \lesssim M(v)^{-1}+M(h)^{-1}$.
A change of variable in $v$ and Fubini's theorem give
\begin{multline*}
\int_{\R^d} \int_{\R^d} \frac{|v+h|^2}{\langle v+h\rangle^2} \langle \rho[f]-u_{\eps\langle
v+h\rangle}[f], \rho[f]\rangle_{L^2_x(\Omega)} M(v)M(h)\,\d h\,\d v\\
\lesssim 
\int_{\R^d} \frac{|v|^2}{\langle v\rangle^2} \langle \rho[f]-u_{\eps\langle
v\rangle}[f], \rho[f]\rangle_{L^2_x(\Omega)} M(v)\left(\int_{\R^d} (M(h)+M(v-h))\,\d h\right)\,\d v
\end{multline*}
so that this term is also controlled by the second term.
In the end, we find
\[
B_{211} \lesssim \frac{1}{\eps^{2s}} \|f^\perp\|_{\HH}
\left(\int_{\R^d}\int_{\Omega}
\left|(-\Delta_v)^s\left[\eps v\cdot\nabla_{\!x}\, u_{\eps\langle v\rangle}[f]\right]\right|^2 M(v)\,\d
x\,\d v\right)^{\frac12}
\lesssim \frac{1}{\eps^s} \|f^\perp\|_{\HH}\sqrt{G[f]}\,.
\]
Similar but easier computations lead to the same estimate for $B_{212}$, so that
\[
B_{21} \lesssim \frac{1}{\eps^{s}} \|f^\perp\|_{\HH}\sqrt{G[f]}\,.
\]
For $B_{22}$ we use the same method, replacing $\rho[f]$ by $R[f^\perp]$, leading to
\[
\begin{aligned}
B_{22}
&\lesssim 
\frac{1}{\eps^{2s}} \|f^\perp\|_{\HH}
\left(\int_{\R^d}\int_{\Omega}
\left|(-\Delta_v)^s\left[\eps v\cdot\nabla_{\!x}\,(\id-\eps^2\langle v\rangle^2\Delta_x)^{-1}R[f^\perp]\right]\right|^2 M(v)\,\d
x\,\d v\right)^{\frac12}\\
&\lesssim \frac{1}{\eps^{2s}} \|f^\perp\|_{\HH}\left(\int_{\R^d} \frac{|v|^2}{\langle v\rangle^2} \|R[f^\perp]\|_{L^2_x(\Omega)}^2 M(v)\,\d v\right)^{\frac12}\\
&\lesssim \frac{1}{\eps^{2s}} \|f^\perp\|_{\HH}^2\,,
\end{aligned}
\]
where we used additionally \cref{eq:tilderho3} for the last inequality.
In both cases $L=L_1$ and $L=L_2$, we thus obtain the estimate
\begin{equation}\label{eq:estimB2}
B_2 \lesssim 
\frac{1}{\eps^{s}} \|f^\perp\|_{\HH}\sqrt{G[f]}
+\frac{1}{\eps^{2s}} \|f^\perp\|_{\HH}^2\,.
\end{equation}
Combining \cref{eq:estimB11,,eq:estimB12,,eq:estimB2} and using Young's inequality, we finally get
\begin{multline*}
\langle  \eps v\cdot \nabla_{\!x} (\id-\eps^2\langle v\rangle^2\Delta_x)^{-1} R[f] M, \Lambda^\eps
f \rangle_{\HH}\\
\leq -\kappa G[f]+C\left(\frac{1}{\eps^{2s}}\|f^\perp\|_{\HH}^2
+\frac{1}{\eps^{2s-1}}\left\|\sqrt{\alpha(2-\alpha)} D^\perp
f_+\right\|_{\partial \HH_+}^2 \right).
\end{multline*}
Let us comment on the gain term $G[f]$. Since $\rho[f]$ is zero-averaged, using~\eqref{def:M1},~\eqref{eq:bound-M2} and~\cref{eq:ellipticetau} and making the change of variable
$w= \sqrt{\mu_1(\Omega)} \eps v$, we have
\[
\begin{aligned}
G[f]
&= \frac{1}{\eps^{2s}}
\int_{\R^d} \langle \rho[f]-u[f],\rho[f]\rangle_{L^2_x(\Omega)}\,\frac{|v|^2}{\langle v\rangle^2} M(v)\, \d v\\
&\geq \frac{1}{\eps^{2s}}\|\rho[f]\|_{L^2_x(\Omega)}^2
\int_{\R^d} \frac{\mu_1(\Omega) |\eps v|^2}{1+\mu_1(\Omega) \eps^2\langle v\rangle^2} M(v)\,\d v\\
& \gtrsim \, \mu_1(\Omega)^s\|\rho[f]\|_{L^2_x(\Omega)}^2
\int_{\R^d} \frac{|w|^2}{(1+\eps^2\mu_1(\Omega)+|w|^2)(\mu_1(\Omega)
\eps^2+|w|^2)^{\frac{d+2s}{2}}}\, \d w\,.
\end{aligned}
\]
Recall that $\eps< \eps_0$, so that
\[
G[f] \gtrsim \mu_1(\Omega)^s \|\rho[f]\|^2_{L^2_x(\Omega)} \gtrsim \d_\Omega^{-2s} \|\rho[f]\|^2_{L^2_x(\Omega)}
\]
where we used the fact that
$ \mu_1(\Omega)=\mu_1(\widehat\Omega) \,\d_\Omega^{-2}$.
This concludes the proof.
\end{proof}

\medskip
\noindent {\it End of the proof of~\cref{theo:main}.} 
We are now able to complete the proof of~\cref{theo:main} by gathering the estimates coming from
\cref{lem:coercivity,,lem:A}, together with \cref{lem:B}. Keeping only the dominant terms, these results imply that for some $\tilde\kappa>0$ and 
$\eps\lesssim  \d_\Omega$, we have
	\begin{align*}
	\la\!\la \Lambda^\eps f, f \ra\!\ra_\eps
	&\leq - \frac{\lambda_0}{\eps^{2s}} \|f^\perp\|^2_\HH - \frac{1}{2\eps^{2s-1}} \Big\|\sqrt{\alpha(2-\alpha)} D^\perp f_+\Big\|^2_{\partial \HH_+} \\
	&\quad
	-\tilde\kappa\delta \d_\Omega^{-2s}\|\rho[f]\|^2_{L^2_x(\Omega)} + \frac{C \delta}{\eps^{2s}} \|f^\perp\|^2_\HH +
	\frac{C\delta}{\eps^{2s-1}} \left\|\sqrt{\alpha(2-\alpha)} D^\perp f_+\right\|_{\partial \HH_+}^2\,.
	\end{align*}
To conclude, choose $\delta>0$ small enough, use the fact that 
	\[
	\|f\|^2_\HH = \|\rho[f]\|^2_{L^2_x(\Omega)} + \|f^\perp\|^2_\HH,
	\]
the equivalence of the norms $\Nt \cdot \Nt_\eps$ and $\|\cdot\|_\HH$ proven
in~\cref{lem:eqnorm} and also the fact that~$\eps
\lesssim \d_\Omega$.
\qed

\subsection{The case of the torus} \label{sec:torus}

Similarly to what is done in \cref{subsec:elliptic}, we can define the operator $(\id-\eps^2\langle
v\rangle^2\Delta_x)^{-1}$ when~$\Omega$ is the torus~$\T^d_\ell$ of size $\ell>0$. The only difference is that~$\mu_1(\Omega)$ now refers to the first nontrivial eigenvalue of the Laplacian on the torus and~${\rm d}_\Omega$ is replaced by $\ell$. The rest of the proof is the same, except that there is no boundary term in the proofs of~\cref{lem:coercivity,lem:A} and of \cref{lem:B}. We can thus conclude in the exact same manner that the conclusion of~\cref{theo:main} also holds true when $\Omega = \T^d_\ell$.

\subsection{The case of the whole space} \label{sec:Rd}
In this section, we prove \cref{theo:wholespace}, which extends the corresponding result of~\cite{BDL} to position-dependent collision operators and furthermore gives a result which is uniform in the anomalous diffusion limit.
 We adapt the estimates to $\Omega = \R^d$ and prove \cref{theo:wholespace}. Let us start with the preliminary results of~\cref{sec:prelim} and observe that $u_\sigma$ can now be defined directly through its Fourier transform, for any~$\eta_1\in L^2(\R^d)$ and~$\eta_2\in (L^2(\R^d))^d$, by the formula
\begin{equation}\label{eq:def_usig_Rd}
    \widehat{u_\sigma}(\xi)\coloneqq\frac{\widehat{\eta_1}(\xi) + i\xi\cdot\widehat{\eta_2}(\xi)}{1+\sigma^2|\xi|^{2}}\,, \quad \forall \xi \in \R^d\,.
\end{equation}
It is easy to see from this formula and the Plancherel identity that \cref{eq:ellipticlapu},
\cref{eq:ellipticgradu2},  and~\cref{eq:ellipticdiffnabla} hold as is, while \cref{eq:ellipticeta2} holds with constant $1$ and no boundary term, and 
\cref{eq:elliptichessiu}, \cref{eq:ellipticdiffhess}, and \cref{eq:ellipticdiffd2sigma} hold with 
universal constants. Finally, \cref{eq:ellipticgradu} and \cref{eq:ellipticetau} hold with $\mu_1(\Omega) = 0$. From there, \cref{lem:eqnorm} and \cref{lem:A} hold. 
A straightforward adaptation of the proof of~\cref{lem:coercivity} and \cref{lem:B} by discarding the boundary contributions yields respectively
	\begin{equation} \label{eq:coercivity2_whole_space}
	\langle \Lambda^\eps g, g \rangle_{\HH} \leq - \frac{\lambda_0}{\eps^{2s}} \|g^\perp\|^2_\HH \,, 
	\end{equation}
    and
	\begin{equation} \label{eq:B_whole_space}
	\langle  \eps v\cdot \nabla_{\!x} (\id-\eps^2\langle v\rangle^2\Delta_x)^{-1} R[f] M, \Lambda^\eps f \rangle_{\HH}  
	\leq
	- \kappa G[f] + \frac{C}{\eps^{2s}} \|f^\perp\|^2_\HH\,,
	\end{equation}
with still the same gain term \cref{eq:defgain}. However, the lower bound on the gain term degenerates~($\mu_1 = 0$) and needs a new estimate in this case. Following \cite{BDL}, we introduce
\begin{equation}
    Q_{2s}[f] \coloneqq \int_{\R^d}\frac{|\xi|^{2s}}{(1+|\xi|^2)^s}|\widehat{\rho}(\xi)|^2\mathrm{d} \xi,
\end{equation}
where $\widehat{\rho}$ denotes the Fourier transform of $\rho[f]$ in the $x$ variable with a slight abuse of notation.

\begin{lem}\label{lem:compare_gain_Fourier}
For a constant depending only on $d$, $s$ and $\eps_0$ one has for any $\eps \in (0,\eps_0)$,
\[
Q_{2s}[f] \lesssim G[f]\,.
\]
\end{lem}
\begin{proof}
By Plancherel formula, \eqref{eq:defgain}, \eqref{eq:def_usig_Rd} with $\eta_1 = \rho[f]$ and $\eta_2=0$, using~\eqref{def:M1} and~\eqref{eq:bound-M2}, one has 
\begin{align*}
G[f] &= \frac{1}{\eps^{2s}}\iint_{\R^d\times\R^d}\frac{\eps^2\langle v\rangle^2|\xi|^2}{1+\eps^2\langle v\rangle^2|\xi|^2}\frac{|v|^2}{\langle v\rangle^2}M(v)|\widehat{\rho}(\xi)|^2\,\mathrm{d} v\,\mathrm{d}\xi\\
&\gtrsim \frac{1}{\eps^{2s}}\iint_{\R^d\times\R^d}\frac{\eps^2|v|^2|\xi|^2}{1+\eps^2|\xi|^2 + \eps^2|v|^2|\xi|^2}(1+|v|^2)^{-\frac{d+2s}{2}}|\widehat{\rho}(\xi)|^2\,\mathrm{d} v\,\mathrm{d}\xi\\
&\gtrsim\frac{1}{\eps^{2s}}\iint_{|v|\geq \max\left(1,\frac{1}{\eps|\xi|}\right)}|v|^{-(d+2s)}|\widehat{\rho}(\xi)|^2\,\mathrm{d} v\,\mathrm{d}\xi\\
&\gtrsim\int_{\R^d}\frac{|\xi|^{2s}}{\max(1, \eps|\xi|)^{2s}}| \widehat{\rho}(\xi)|^2\,\mathrm{d}\xi\,,
\end{align*}
with constants depending only on $d$ and $s$. Since $\eps\in (0,\eps_0)$ one obtains the desired result.
\end{proof}
The next lemma is a direct consequence of~\cite[Lemma~13]{BDL} and \cref{lem:compare_gain_Fourier}.
\begin{lem}\label{lem:PhiG}
For a constant depending only on $d$ and $s$ and $\eps_0$ one has for any $\eps \in (0,\eps_0)$, 
\[
\|{\rho}[f]\|^2_{L^2(\R^d)}\lesssim \|{\rho}[f]\|_{L^1(\R^d)}^2\Phi_{\theta}\left(\frac{G[f]}{\|{\rho}[f]\|_{L^1(\R^d)}^2}\right)
\]
where $\Phi_{\theta}(z) \coloneqq z^{\theta}(1+z^{1-\theta})$ and $\theta\coloneqq {\frac{d}{d+2s}}$.
\end{lem}
We can now turn to the proof of~\cref{theo:wholespace}. 
\begin{proof}[Proof of~\cref{theo:wholespace}] From \cref{eq:coercivity2_whole_space}, \cref{eq:B_whole_space} and \cref{lem:A}, we can choose $\delta$ small enough so that
\begin{equation} \label{eq:hypoRd0}
	\begin{aligned}
	\la\!\la \Lambda^\eps f, f \ra\!\ra_\eps
	&\lesssim - \frac{1}{\eps^{2s}} \|f^\perp\|^2_\HH - G[f]\,.
	\end{aligned}
\end{equation}
If $\rho[f]=0$, then $G[f]=0$, $f^\perp=f$ so that we have in fact
\[
\la \! \la \Lambda^\eps f , f \ra\!\ra_{\eps} 
\lesssim - \frac{1}{\eps^{2s}} \|f\|_{\HH}^2\,,
\]
and the theorem is proven. Otherwise, using the fact that
$G[\lambda f]=\lambda^2 G[f]$ and the homogeneity of the other terms of~\cref{eq:hypoRd0}, up to considering $f/\|\rho[f]\|_{L^1(\R^d)}$ in place of $f$ we may
assume $\|\rho[f]\|_{L^1(\R^d)}=1$  and we need only show
\begin{equation}\label{PhiG:restmp}
\la \! \la \Lambda^\eps f , f \ra\!\ra_{\eps} 
\lesssim - \Psi_\theta\left(\Nt f\Nt_{\eps}^2\right) - \frac{1}{\eps^{2s}} \| f^\perp \|_\HH^2 \,,
\end{equation}
where we recall that $\frac1\theta = 1 + {\frac{2s}{d}}$ and define
\[
\Psi_{\theta}(z) \coloneqq \frac{z^{\frac1\theta}}{1+z^{\frac1\theta-1}}\,.
\]
By \cref{lem:PhiG}, estimating the inverse of $\Phi_{\theta}$ we deduce that
\begin{equation} \label{eq:hypoRd}
	\la\!\la \Lambda^\eps f, f \ra\!\ra_\eps
	\lesssim - \Psi_\theta\left(\|\rho[f]\|_{L^2(\R^d)}^2\right)
- \frac{1}{\eps^{2s}} \|f^\perp\|^2_\HH \,.
\end{equation}
We now distinguish two cases. If $\|f^\perp\|_{\HH}^2\leq \|f\|_\HH^2/2$, then noticing
that $\Psi_\theta$ is non-decreasing, we estimate
\[
\Psi_\theta\left(\|\rho\|_{L^2(\R^d)}^2\right)
= \Psi_\theta\left(\|f\|_{\HH}^2-\|f^\perp\|_{\HH}^2\right)
\geq \Psi_\theta(\|f\|_{\HH}^2/2)
\]
thus
\[
\la\!\la \Lambda^\eps f, f \ra\!\ra_\eps
\lesssim- \Psi_\theta\left(\|f\|_\HH^2/2\right)
 - \frac{1}{\eps^{2s}} \|f^\perp\|^2_\HH \,.
\]
If on the contrary $\|f^\perp\|_{\HH}^2>\|f\|_\HH^2/2$, noticing that 
$\Psi_\theta(t)\leq t$ and using $\eps\leq \eps_0$, we estimate
\[
-\frac{1}{2\eps^{2s}}\|f^\perp\|_\HH^2
< -\frac{1}{4\eps^{2s}} \|f\|_{\HH}^2
\leq -\frac{1}{2 \eps_0^{2s}} \Psi_\theta\left(\|f\|_{\HH}^2/2\right),
\]
from which, with \cref{eq:hypoRd}, we deduce that
\[
\la\!\la \Lambda^\eps f, f \ra\!\ra_\eps
\lesssim - \frac{1}{\eps^{2s}} \|f^\perp\|^2_\HH
\leq -\frac{1}{2\eps_0^{2s}} \Psi_\theta\left(\|f\|_{\HH}^2/2\right)
-\frac{1}{2\eps^{2s}} \|f^\perp\|^2_\HH \,.
\]
Gathering both cases and using the inequality $\Psi_\theta(t/2) \geq \Psi_\theta(t)/2^{\frac{1}{\theta}}$
yields
\[
\la\!\la \Lambda^\eps f, f \ra\!\ra_\eps
\lesssim- \Psi_\theta\left(\|f\|_{\HH}^2\right) - \frac{1}{\eps^{2s}} \|f^\perp\|^2_\HH .
\]
Since $\|\cdot\|_{\HH}$ and $\Nt\cdot\Nt_\eps$ are equivalent uniformly in $\eps\in(0,\eps_0)$, we have $\Psi_\theta\left(\|f\|_{\HH}^2\right)\gtrsim \Psi_\theta\left(\Nt
f\Nt_\eps^2\right)$ which yields \cref{PhiG:restmp} and concludes the proof.
\end{proof}

\begin{proof}[Proof of~\cref{cor:wholespace}] Using that $e^{t \Lambda^\eps}$ preserves positivity and mass, we can deduce that $\|e^{t\Lambda^\eps} g\|_{L^1(\R^d\times\R^d)} = \|g\|_{L^1(\R^d\times\R^d)}$ for $g \geq 0$. Then, by splitting $f^\eps_{\mathrm{in}}$ into nonpositive and nonnegative parts, one checks that the $L^1$-norm of $f^\eps(t) = e^{t \Lambda^\eps}f^\eps_{\mathrm{in}}$ is nonincreasing along time.
 In particular, we have
 \begin{equation} \label{eq:boundrho}
 \|\rho[f^\eps]\|_{L^1(\R^d)}
 \leq \|f^\eps\|_{L^1(\R^d\times\R^d)} \leq \|f^\eps_{\mathrm{in}}\|_{L^1(\R^d\times \R^d)}\,.
 \end{equation}
As a consequence, we can deduce that the conclusion of~\cref{theo:wholespace} holds true and using~\eqref{eq:boundrho}, we get
\[
\la \! \la \Lambda^\eps f^\eps , f^\eps \ra\!\ra_{\eps} 
\lesssim - \Nt f^\eps \Nt_\eps^{2}\left(\frac{\Nt f^\eps \Nt_\eps^{\frac{4s}{d}}}{\|f^\eps_{\mathrm{in}}\|_{L^1(\R^d\times\R^d)}^{\frac{4s}{d}}
+\Nt f^\eps
\Nt_\eps^{\frac{4s}{d}}}\right)
- \frac{1}{\eps^{2s}} \|(f^\eps)^\perp\|_{\HH}^2\,.
\]
Consequently, setting $y(t)\coloneqq \Nt f^\eps\Nt_\eps^2$, $\ell_1\coloneqq\|f^\eps_{\mathrm{in}}\|_{L^1(\R^d\times\R^d)}^2$,
$g\coloneqq \frac{1}{\eps^{2s}}\|(f^\eps)^\perp\|_\HH^2$, $a\coloneqq 2s/d$ and
\[
F_a(z)\coloneqq \frac{z^{1+a}}{\ell_1^a+z^a}\,,
\]
we have
\[
y'(t)\lesssim -(F_a(y(t))+g(t))\,.
\]
The conclusion of~\cref{cor:wholespace} then follows by elementary arguments, which we recall in
the next lemma for the reader's convenience.
\end{proof}

\begin{lem}\label{lem:classical}
Let $a>0$, $y:[0,\infty)\to [0,\infty)$ be locally absolutely continuous, and let
$g:[0,\infty)\to[0,\infty)$ be locally integrable. Assume that for some $\kappa,\ell_1>0$, we have
\begin{equation} \label{eq:diffineq}
y'(t) \leq -\kappa (F_a(y(t)) + g(t)) \,,\quad \forall t>0\,,\quad \text{where}\quad F_a(z)\coloneqq \frac{z^{1+a}}{\ell_1^a+z^a}\,.
\end{equation}
Then, for all $\beta\in[0,\frac1a)$ one has for any $t\geq0$,
\begin{equation}\label{eq:dissipineq}
y(t)\leq \frac{y(0)}{(1+ \mu_0 t)^{\frac{1}{a}}}\qquad\text{and}\qquad
\kappa(1-\beta a)\int_{0}^t g(\tau)(1+\mu_0 \tau)^\beta\,\mathrm{d}\tau
\leq y(0)\,,
\end{equation}
with 
\[
\mu_0\coloneqq \frac{\kappa a}{1+\left(\frac{\ell_1}{y(0)}\right)^{a}}\,.
\]
\end{lem}

\begin{proof} Observe that $y(t)\leq y(0)$ and without loss of generality we assume $y(0)>0$ and $y(t)>0$ for all $t\geq0$. Then one has
\[
(y^{-a})'(t)\geq \frac{a\kappa}{\ell_1^a+y(0)^a}\,, \quad \forall t>0\,.
\]
Integrating the inequality 
we deduce
\begin{equation}\label{eq:boundinfty}
y(t)\leq \frac{y(0)}{\left(1+\mu_0 t \right)^{\frac{1}{a}}}\,, \quad \forall t \geq 0\,.
\end{equation}
Then, let $W(\tau) =(1+\mu_0 \tau)^{\beta}$.  Using the differential inequality~\eqref{eq:diffineq} and integration by parts, one gets
 \[
 \kappa\int_0^{t}W(\tau)(F_a(y(\tau))+ g(\tau))\,\mathrm{d}\tau \leq \int_0^{t} W'(\tau)y(\tau)\,\mathrm{d}\tau +  y(0)
 \leq \frac{ y(0)}{1-\beta a}\,,
 \]
 where the last bound is obtained after using \cref{eq:boundinfty}. This ends the proof of~\eqref{eq:dissipineq}.
\end{proof}

\appendix

\section{Scale-invariant trace and related inequalities} \label{app:trace}

We recall the multiplicative trace inequality that can be found, e.g., in
\cite[Theorem~1.5.1.10]{Gri2011}. We provide a proof which makes explicit the dependence in the size
of the domain for the sake of completeness (recall the notation $\d_\Omega \coloneqq \mathrm{diam}(\Omega)$ for any nonempty~$\Omega \subset \R^d$).

\begin{prop}\label{prop:multtraceineq}
Let $\Omega\subset\R^d$ be a bounded open set with Lipschitz boundary.
Then there exists a constant $C(\Omega)$ such that $C(\lambda\Omega)=C(\Omega)$ for every
$\lambda>0$ and 
\[
\|u\|_{L^2(\partial\Omega)}^2 \leq C(\Omega) 
\|u\|_{L^2(\Omega)} \left( \d_\Omega^{-1}\|u\|_{L^2(\Omega)}+\|\nabla
u\|_{L^2(\Omega)}\right),
\qquad\forall u\in H^1(\Omega)\,.
\]
\end{prop}

\begin{proof}
Let $\widehat \Omega = \Omega \, \d_\Omega^{-1}$, so that $\d_{\widehat \Omega}=1$.
By \cite[Lemma~1.5.1.9]{Gri2011} there exists a vector field~$\mu \in \CC^\infty\left(\overline{\widehat \Omega}; \R^d\right)$ and
$\delta>0$ such that
\[
\mu(x)\cdot n(x) \geq \delta\,,\qquad\text{for $\mathscr{H}^{d-1}$-a.e.} \,\, x\in\partial\widehat\Omega\,.
\]
Let $u\in H^1(\widehat \Omega)$. By the divergence theorem, we have
\begin{multline*}
\delta\|u\|_{L^2(\partial\widehat \Omega)}^2\leq \int_{\partial\widehat \Omega} |u(x)|^2 (\mu(x)\cdot n(x)) \,\d\mathscr{H}^{d-1}(x) 
= \int_{\widehat \Omega} (2 u \nabla u\cdot \mu + |u|^2\nabla\cdot \mu) \,\d x\\
\leq 2 \|\mu\|_{L^\infty(\widehat \Omega)}\|u\|_{L^2(\widehat \Omega)}\|\nabla u\|_{L^2(\widehat \Omega)}
+ \|\nabla\cdot\mu\|_{L^\infty(\widehat \Omega)}\|u\|^2_{L^2(\widehat \Omega)}\,.
\end{multline*}
Scaling back, for any $u\in H^1(\Omega)$ we get
\[
\|u\|_{L^2(\partial\Omega)}^2
\leq \delta^{-1} \|u\|_{L^2(\Omega)}\left(2 \|\mu\|_{L^\infty(\widehat \Omega)}\|\nabla u\|_{L^2(\Omega)}
+ \mathrm{diam}(\Omega)^{-1}\|\nabla\cdot\mu\|_{L^\infty(\widehat \Omega)}\|u\|_{L^2(\Omega)}\right).
\]
\end{proof}

We also recall the classical estimate $\|\nabla^2 u\|_{L^2(\Omega)}\lesssim \|\Delta
u\|_{L^2(\Omega)}$ in the case of a $\CC^2$ bounded domain with Neumann boundary condition.

\begin{prop}\label{prop:hessbylap}
Let $\Omega\subset\R^d$ be an open bounded domain with $\CC^2$ boundary. Then there exists a
constant $C(\Omega)$ with $C(\lambda\Omega)=C(\Omega)$ for every
$\lambda>0$ such that the following holds.
For every $u\in H^2(\Omega)$ satisfying $n\cdot\nabla u=0$ on $\partial\Omega$, we have
\begin{equation}\label{hessbylap:res}
\|\nabla^2 u\|_{L^2(\Omega)} \leq C(\Omega)\|\Delta u\|_{L^2(\Omega)}\,.
\end{equation}
Besides, if $\Omega$ is convex, then \cref{hessbylap:res} holds with $C(\Omega)=1$.
\end{prop}

\begin{proof}
Up to subtracting its average, without loss of generality we may assume that $u$ is zero-averaged.
Since $\|\nabla^2 u\|_{L^2(\Omega)}$ and $\|\Delta u\|_{L^2(\Omega)}$ behave the same under
rescaling, the best constant in \cref{hessbylap:res} is necessarily invariant by rescaling. Thus we
need only show that the inequality holds for some constant $C(\Omega)$. For the rest of the proof, we thus do not specify the dependence of our constants on $\d_\Omega$ and the constant $C(\Omega)$ is not necessarily invariant under scaling.
Applying \cite[Theorem 3.1.1.1]{Gri2011} with $v=\nabla u\in H^1(\Omega)$ and using the boundary
condition, we find
\[
\| \Delta u\|_{L^2(\Omega)}^2 = 
\| \nabla^2 u\|_{L^2(\Omega)}^2 + \int_{\partial\Omega} \Pi(x)(\nabla_T\, u,\nabla_T\, u)\,\d\mathscr{H}^{d-1}(x)\,,
\]
where $\nabla_T\, u$ is the tangential gradient of $u$ on $\partial\Omega$, and $\Pi$ is the second
fundamental form of the manifold $\partial\Omega$.
If $\Omega$ is convex, then $\Pi$ is a nonnegative quadratic form and we immediately get
\cref{hessbylap:res} with $C(\Omega)=1$. Otherwise, since $\Omega$ is a bounded domain with
$\CC^2$ boundary, there exists $\kappa>0$ such that
\[
\Pi(\nabla_T\, u,\nabla_T\, u) \geq  -\kappa |\nabla_T\, u|^2.
\]
Thus
\begin{equation}\label{hessbylap:eq1}
\| \Delta u\|_{L^2(\Omega)}^2 \geq 
\| \nabla^2 u\|_{L^2(\Omega)}^2 -\kappa\int_{\partial\Omega} |\nabla_T\, u|^2\,\d\mathscr{H}^{d-1}\,.
\end{equation}
By the multiplicative trace inequality \cref{prop:multtraceineq}, we have
\begin{equation}\label{hessbylap:eq2}
\| \nabla u\|_{L^2(\partial\Omega)}^2 
\leq C(\Omega) \|\nabla u\|_{L^2(\Omega)}\left(\|\nabla u\|_{L^2(\Omega)}+\|\nabla^2
u\|_{L^2(\Omega)}\right).
\end{equation}
Then integrating by parts and using the Cauchy--Schwarz and Poincaré--Wirtinger inequalities we estimate
\[
\|\nabla u\|_{L^2(\Omega)}^2 = \langle -\Delta u, u\rangle_{L^2(\Omega)}
\leq C(\Omega) \|\Delta u\|_{L^2(\Omega)}\| \nabla u\|_{L^2(\Omega)}\,,
\]
hence
\begin{equation}\label{hessbylap:eq3}
\|\nabla u\|_{L^2(\Omega)} \leq C(\Omega) \|\Delta u\|_{L^2(\Omega)}\,.
\end{equation}
We conclude the proof by gathering \cref{hessbylap:eq1,,hessbylap:eq2,,hessbylap:eq3} and using
a suitable Young's inequality.
\end{proof}

\section{Alternative approach in bounded domains} \label{app:alternative}
In this appendix, we present an alternative proof of our main exponential stability result in the bounded-domain case. It is worth mentioning that, unlike the approach used in the body of the paper, this one is not robust enough to yield the whole-space case as a consequence of the analysis.
In this alternative approach, new macroscopic quantities are introduced, namely
	\begin{equation} \label{def:twisted}
	R^\eps [g] (x) \coloneqq \int_{\R^d} g (x,v) \, \frac{\d v}{\langle \eps v \rangle^2} 
	\qquad \text{and} \qquad 
	J^\eps [g] (x)\coloneqq \int_{\R^d} g (x,v)\, v \,\frac{\d v}{\langle \eps v \rangle^2} \,. 
	\end{equation}
Observe that with these definitions, at least formally $R^\eps\to\rho$ and $J^\eps\to j$ as $\eps\to0$. As previously, the additional weight of order $-2$  counterbalances the lack of integrability of our heavy-tailed distributions. 
Let us state some preliminary estimates on the twisted macroscopic quantities (comparable to the ones stated in \cref{lem:macro}) whose proofs are straightforward using the Cauchy--Schwarz inequality and the change of variable $v \to \eps v$.
\begin{lem}
The macroscopic quantities $R^\eps [g]$ and $J^\eps [g]$ satisfy the following identities:
    \begin{align*} 
	R^\eps [g](x)& = c_\eps \rho[g](x) + \int_{\R^d} g^\perp(x,v) \, \frac{\d v}{\langle \eps v \rangle^2}\,,\quad\text{ with  }c_\eps \coloneqq \int_{\R^d} M(v)  \, \frac{\d v}{\langle \eps v \rangle^2} \xrightarrow[\eps \to 0]{} 1\,, \\
	J^\eps[g] &= J^\eps[g^\perp]\,.
 	\end{align*}
They additionally satisfy the following estimates:
	\begin{align*}
	\left\|R^\eps[g]\right\|_{L^2_x(\Omega)} &\leq \left\|g\right\|_{\HH}\,,\\  
 	\left\|J^\eps[g]\right\|_{L^2_x(\Omega)} &\lesssim \eps^{s-1} \|g^\perp\|_\HH\,,\\ 	
 	\left\| \int_{\R^d} v_i v_j g^\perp \, \frac{\d v}{\langle \eps v \rangle^2} \right\|_{L^2_x(\Omega)} 	&\lesssim \eps^{s-2} \|g^\perp\|_\HH \,.
 	\end{align*}
\end{lem}

We introduce a new bilinear form 
on $\HH$ defined for $\eps \in (0,1]$ through
	\begin{multline} \label{def:scalar2}
	\langle\!\langle f, g \rangle\!\rangle_\eps \coloneqq 
	\langle f, g \rangle_\HH  
	+ \eps\delta \langle \nabla_{\!x} (\id-\Delta_x)^{-1} R^\eps[f], J^\eps[g] \rangle_{L^2_x(\Omega)} \\
	+ \eps\delta \langle \nabla_{\!x} (\id-\Delta_x)^{-1} R^\eps[g], J^\eps[f] \rangle_{L^2_x(\Omega)}
	\end{multline}
where $\delta>0$ will be chosen small enough and $(\id-\Delta_x)^{-1} R^\eps[\cdot]$ is defined thanks to the results recalled in \cref{subsec:elliptic} with $\eta_1 =  R^\eps[\cdot] $, $\eta_2=0$ and $\sigma = 1$. 
We denote by $\Nt\cdot \Nt_\eps$ the associated norm which writes as follows
	\begin{equation} \label{def:Nt2}
	\Nt g \Nt^2_\eps \coloneqq 
	\|g\|^2_\HH + 2\eps\delta \langle \nabla_{\!x} (\id-\Delta_x)^{-1} R^\eps[g], J^\eps[g] \rangle_{L^2_x(\Omega)}\,.
	\end{equation}
As in~\cref{lem:eqnorm}, we have the following uniform-in-$\eps$ equivalence of the norms:
\begin{lem}
The norm $\Nt \cdot \Nt_\eps$ defined in~\eqref{def:Nt2} and the classical norm $\|\cdot \|_\HH$ on $\HH$ are equivalent for $\delta$ small enough {\em uniformly in $\eps \in (0,1]$}. 
\end{lem}
In the end, using similar (but less refined) estimates as the ones presented in \cref{subsec:proof}, we can prove the following less precise version of~\cref{theo:main}:
\begin{theo}
Consider $\eps \in (0,1]$. The scalar product and the associated norm respectively defined in~\eqref{def:scalar2} and~\eqref{def:Nt2} are such that the linear operator $\Lambda^\eps$ satisfies the following coercivity estimate: 
$$
\la \! \la \Lambda^\eps f , f \ra\!\ra_{\eps} 
\lesssim -  \Nt f \Nt^2_\eps - \frac{1}{\eps^{2s}} \|f^\perp\|_{\HH}^2\,, 
$$
for any $f \in \mathrm{Dom}(\Lambda^\eps)$ satisfying the boundary condition in~\eqref{eq:dtf=Lf} and $\int_{\Omega \times \R^d} f(x,v) \, \d x \, \d v =0$, where the multiplicative constant does not depend on $\eps \in (0,1]$. 
\end{theo}


\medskip
 \noindent\textbf{Acknowledgments.} 
 MH and MP acknowledge the support of the CDP C2EMPI, together with the French State under the France-2030 programme, the University of Lille, the Initiative of Excellence of the University of Lille, the European Metropolis of Lille for their funding and support of the R-CDP-24-004-C2EMPI project. MP is partially supported by the ANR projects STOIQUES (ANR-24-CE40-2216) and ESSTOS (ANR-25-CE40-4799). IT was supported by the French government through the France 2030 investment plan managed by the ANR, as part of the Initiative of Excellence Universit\'e  Côte d'Azur under reference number ANR-15-IDEX-01 as well as by the ANR project MaDynOS (ANR-24-CE40-3535-01).

\medskip
\noindent\textbf{Declaration of AI use.} 
 The authors used the AI assistant Claude (Anthropic) to proofread the finalized version of the manuscript (check the mathematical arguments,  suggest minor corrections and rewording); all suggestions were verified by the authors. The authors assume responsibility for all content.

\bigskip
\bibliographystyle{acm}
\bibliography{bibli}

\begin{thebibliography}{10}

\bibitem{AHHT1}
{\sc Ayi, N., Herda, M., Hivert, H., and Tristani, I.}
\newblock A note on hypocoercivity for kinetic equations with heavy-tailed
  equilibrium.
\newblock {\em C. R. Math. Acad. Sci. Paris 358}, 3 (2020), 333--340.

\bibitem{AHHT2}
{\sc Ayi, N., Herda, M., Hivert, H., and Tristani, I.}
\newblock On a structure-preserving numerical method for fractional
  {F}okker-{P}lanck equations.
\newblock {\em Math. Comp. 92}, 340 (2023), 635--693.

\bibitem{Bardos}
{\sc Bardos, C.}
\newblock Probl\`emes aux limites pour les \'{e}quations aux d\'{e}riv\'{e}es
  partielles du premier ordre \`a coefficients r\'{e}els; th\'{e}or\`emes
  d'approximation; application \`a l'\'{e}quation de transport.
\newblock {\em Ann. Sci. \'{E}cole Norm. Sup. (4) 3\/} (1970), 185--233.

\bibitem{benabdallah2011fractional}
{\sc Ben~Abdallah, N., Mellet, A., and Puel, M.}
\newblock Fractional diffusion limit for collisional kinetic equations: a
  {Hilbert} expansion approach.
\newblock {\em Kinet. Relat. Models 4}, 4 (2011), 873--900.

\bibitem{Bernou}
{\sc Bernou, A.}
\newblock Asymptotic behavior of degenerate linear kinetic equations with
  non-isothermal boundary conditions.
\newblock {\em J. Differ. Equations 442\/} (2025), 52.
\newblock Id/No 113470.

\bibitem{BCMT}
{\sc Bernou, A., Carrapatoso, K., Mischler, S., and Tristani, I.}
\newblock Hypocoercivity for kinetic linear equations in bounded domains with
  general {M}axwell boundary condition.
\newblock {\em Ann. Inst. H. Poincar\'{e} C Anal. Non Lin\'{e}aire 40}, 2
  (2023), 287--338.

\bibitem{BDL}
{\sc Bouin, E., Dolbeault, J., and Lafleche, L.}
\newblock Fractional hypocoercivity.
\newblock {\em Comm. Math. Phys. 390}, 3 (2022), 1369--1411.

\bibitem{corrigendum}
{\sc Cao, C., and Carrapatoso, K.}
\newblock Corrigendum to ``{H}ydrodynamic limit for the non-cutoff {B}oltzmann
  equation''.
\newblock {\em
  https://carrapatoso.perso.math.cnrs.fr/articles/19-boltzmann\_hydro\_corrigendum.pdf}.

\bibitem{CaoCarrapatoso2026}
{\sc Cao, C., and Carrapatoso, K.}
\newblock Hydrodynamic limit for the non-cutoff {Boltzmann} equation.
\newblock {\em Ann. Inst. Henri Poincar{\'e} C, Anal. Non Lin{\'e}aire 43}, 2
  (2026), 417--482.

\bibitem{Carrapatoso-Mischler}
{\sc Carrapatoso, K., and Mischler, S.}
\newblock The kinetic {Fokker}-{Planck} equation in a domain:
  ultracontractivity, hypocoercivity, and long-time asymptotic behavior.
\newblock {\em Atti Accad. Naz. Lincei, Cl. Sci. Fis. Mat. Nat., IX. Ser.,
  Rend. Lincei, Mat. Appl. 35}, 4 (2024), 643--680.

\bibitem{cesbron2018anomalous}
{\sc Cesbron, L.}
\newblock Anomalous diffusion limit of kinetic equations in spatially bounded
  domains.
\newblock {\em Commun. Math. Phys. 364}, 1 (2018), 233--286.

\bibitem{cesbron2022fractional}
{\sc Cesbron, L., Mellet, A., and Puel, M.}
\newblock Fractional diffusion limit of a kinetic equation with diffusive
  boundary conditions in a bounded interval.
\newblock {\em Asymptotic Anal. 130}, 3-4 (2022), 367--386.

\bibitem{Cessenat}
{\sc Cessenat, M.}
\newblock Th\'{e}or\`emes de trace {$L^p$} pour des espaces de fonctions de la
  neutronique.
\newblock {\em C. R. Acad. Sci. Paris S\'{e}r. I Math. 299}, 16 (1984),
  831--834.

\bibitem{dechicha2024fractional}
{\sc Dechicha, D., and Puel, M.}
\newblock Fractional diffusion for {Fokker}-{Planck} equation with heavy tail
  equilibrium: an {\`a} la {Koch} spectral method in any dimension.
\newblock {\em Asymptotic Anal. 136}, 2 (2024), 79--132.

\bibitem{DGP}
{\sc Degond, P., Goudon, T., and Poupaud, F.}
\newblock Diffusion limit for non homogeneous and non-micro-reversible
  processes.
\newblock {\em Indiana Univ. Math. J. 49}, 3 (2000), 1175--1198.

\bibitem{DMS1}
{\sc Dolbeault, J., Mouhot, C., and Schmeiser, C.}
\newblock Hypocoercivity for kinetic equations with linear relaxation terms.
\newblock {\em Comptes Rendus Mathematique 347}, 9-10 (2009), 511 -- 516.

\bibitem{DMS2}
{\sc Dolbeault, J., Mouhot, C., and Schmeiser, C.}
\newblock Hypocoercivity for linear kinetic equations conserving mass.
\newblock {\em Trans. Amer. Math. Soc. 367}, 6 (2015), 3807--3828.

\bibitem{Gentil-Imbert}
{\sc Gentil, I., and Imbert, C.}
\newblock The {L}\'{e}vy-{F}okker-{P}lanck equation: {$\Phi$}-entropies and
  convergence to equilibrium.
\newblock {\em Asymptot. Anal. 59}, 3-4 (2008), 125--138.

\bibitem{Gri2011}
{\sc Grisvard, P.}
\newblock {\em Elliptic {{Problems}} in {{Nonsmooth Domains}}}.
\newblock {Society for Industrial and Applied Mathematics}, 2011.

\bibitem{Herau-AA}
{\sc H\'{e}rau, F.}
\newblock Hypocoercivity and exponential time decay for the linear
  inhomogeneous relaxation {B}oltzmann equation.
\newblock {\em Asymptot. Anal. 46}, 3-4 (2006), 349--359.

\bibitem{Mellet-Indiana}
{\sc Mellet, A.}
\newblock Fractional diffusion limit for collisional kinetic equations: a
  moments method.
\newblock {\em Indiana Univ. Math. J. 59}, 4 (2010), 1333--1360.

\bibitem{mellet2011fractional}
{\sc Mellet, A., Mischler, S., and Mouhot, C.}
\newblock Fractional diffusion limit for collisional kinetic equations.
\newblock {\em Arch. Ration. Mech. Anal. 199}, 2 (2011), 493--525.

\bibitem{Mischler-CPDE}
{\sc Mischler, S.}
\newblock On the trace problem for solutions of the {V}lasov equation.
\newblock {\em Comm. Partial Differential Equations 25}, 7-8 (2000),
  1415--1443.

\bibitem{Pazy}
{\sc Pazy, A.}
\newblock {\em Semigroups of linear operators and applications to partial
  differential equations}, vol.~44 of {\em Appl. Math. Sci.}
\newblock Springer, Cham, 1983.

\bibitem{Tem1979}
{\sc Temam, R.}
\newblock {\em Navier-{{Stokes}} Equations: Theory and Numerical Analysis},
  rev. ed~ed.
\newblock No.~v. 2 in Studies in Mathematics and Its Applications.
  {North-Holland Pub. Co. ; sole distributors for the U.S.A. and Canada,
  Elsevier North-Holland}, 1979.

\bibitem{Tristani-CMS}
{\sc Tristani, I.}
\newblock Fractional {F}okker-{P}lanck equation.
\newblock {\em Commun. Math. Sci. 13}, 5 (2015), 1243--1260.

\end{thebibliography}
\end{document}